\documentclass[12pt,a4]{article}

\usepackage{url}
\usepackage{color}
\usepackage[T1]{fontenc}
\usepackage[english]{babel}
\usepackage{setspace}
\usepackage{csquotes}

\usepackage[margin=1in]{geometry} 
\usepackage{amsmath,amsthm,amssymb,mathrsfs}
\usepackage{kpfonts} 
\usepackage{xcolor}
\usepackage{hyperref}
\hypersetup{
    colorlinks=true,
    linkcolor=blue,
    filecolor=magenta,
    urlcolor=cyan,
    citecolor=blue,
}
\usepackage{enumitem}
\usepackage{graphicx}
\usepackage{subcaption}
\usepackage{float}

\newcommand{\Z}{\mathbb{Z}}
\newcommand{\R}{\mathbb{R}}

\newtheorem{df}{Definition}[section]
\newtheorem{lem}[df]{Lemma}

\newtheorem{pro}[df]{Proposition}
\newtheorem{thm}[df]{Theorem}

\usepackage{fancyhdr}

\addto\captionsenglish{}

\begin{document}

\title{Convergence of a discrete selection-mutation model with exponentially decaying mutation kernel to a Hamilton-Jacobi equation}

\author{\large{Anouar Jeddi$^{\ast}$}\\
\small{$^{\ast}$CMAP, École polytechnique, Institut polytechnique de Paris,}\\
    \text{anouar.jeddi@polytechnique.edu}
}

\maketitle

\newcommand{\K}{\varepsilon}

 \renewcommand{\thesection}{\arabic{section}}
 \renewcommand{\theequation}{\arabic{equation}}

\begin{abstract}
    In this paper we derive a constrained Hamilton-Jacobi equation with obstacle from a discrete non-linear integro-differential model of population dynamics, with exponentially decaying mutation kernel. The exponential decay of the kernel leads to a modification of the classical Hamilton-Jacobi equation obtained previously from continuous models in \cite{BMP}. We consider a population composed of individuals characterized by a quantitative trait, subject to selection, mutation and competition. In a regime of small mutations, small spatial discretization step and large time we prove that the WKB transformation of the density converges to a viscosity solution of a constrained Hamilton-Jacobi equation with obstacle.         
\end{abstract}
\section{Introduction}
In this paper, we are interested in the asymptotic analysis in a regime of small mutations and large time of a discrete version of the non-local integro-differential equation  
\begin{align}
\label{eq:selction-muation}
    \begin{cases}
        \partial_t n(t,x)=R(x,I(t))n(t,x)+\int_{\R}p(y)G(y-x)n(t,y)dy,~\forall(t,x)\in (0,+\infty)\times \R,\\
        I(t)=\int_{\R}n(t,y)dy ~~~\forall t>0,\\
        n(0,.)=n^0(.).
\end{cases}
\end{align}
with a kernel $G(.)$ that has exponential decay. This analysis leads to a non-standard Hamilton-Jacobi equation.\\

The equation \eqref{eq:selction-muation} belongs to a class of models arising in 
 eco-evolutionary dynamics of populations structured by quantitative traits and subject to selection, mutation and competition(see, for instance \cite{OPSB,BMP,PB,LMP,CJ}). Here,  $n(t,x)$ denotes the phenotypic density at time $t$ and trait $x,$ $R(.,.)$ is  the growth rate which depends on traits and the total population size $I,$  $p$ is the mutation rate, and $G$ is a probability kernel of mutation effects. An asymptotic approach, involving Hamilton-Jacobi equations, has indeed been widely developed to study models from evolutionary biology in the small mutational variance regime \cite{OPSB,  BMP,PB}. In these works the Hamilton-Jacobi approach was usually carried out for mutations kernel $G$ with fast decay. Biologically it is interesting to consider a kernel with  slow decay (see \cite{BW}), since it allows us to take into account large mutation jumps with a high rate. In this work, we extend the previous approach to kernels with slow decay.\\   
 
 We briefly recall the standard analysis leading to Hamilton-Jacobi equations. A particular regime of interest is when the mutations have small effects. To take small mutational effects into account, one replaces $G(.)$ by $\frac{1}{\varepsilon}G(\frac{.}{\varepsilon})$, considering a mutational variance of order $\varepsilon
^2$. A small mutational variance leads to a small phenotypic variance in the population, and hence to slow dynamics of the phenotypic density. In order to capture the effects of mutations in such dynamics one accelerates the time by replacing $t$ by $t/\varepsilon$. Then, \eqref{eq:selction-muation} becomes 
\begin{align}
\label{eq:scaling selction-muation}
\begin{cases}
        \partial_t n_\varepsilon(t,x)=\frac{1}{\varepsilon}R(x,I_\varepsilon(t))n_\varepsilon(t,x)+\frac{1}{\varepsilon}\int_{\R}p(y)\frac{1}{\varepsilon}G\big(\frac{y-x}{\varepsilon}\big)n_\varepsilon(t,y)dy,~~~~\forall(t,x)\in (0,+\infty)\times \R,\\
        I_\varepsilon(t)=\int_{\R}n_\varepsilon(t,y)dy, ~~~\forall t>0,\\
    n_\varepsilon(0,.)=n^{\K,0}(.).
        \end{cases}
\end{align}
The solution $n_\varepsilon$ to such an equation typically concentrates around an evolving dominant trait $\bar x(t)$. In order to characterize such concentration phenomenon  it is convenient to perform a Hopf-Cole transformation $u_\varepsilon(t,x)=\varepsilon \log{n_\varepsilon(t,x)}.$ Such a transformation unfolds the singularity of the phenotypic density which is close to a Dirac mass, since $u^\varepsilon$ will be close to a continuous function, solving a Hamilton-Jacobi equation. It is indeed proved in \cite{PB,BMP} assuming that $G$ has super exponential decay at infinity, that as $\varepsilon\to 0$, $u^\varepsilon$ converges locally uniformly to the viscosity solution of the following Hamilton-Jacobi equation with constraint.
\begin{equation}
\label{eq:constrained HJ}
    \begin{cases}
        \partial_t u(t,x)=R(x,I(t))+p(x)\int_{\R}G(y)e^{\nabla u(t,x).y}dy,~~~~(t,x)\in (0,+\infty)\times \R,\\
        \max_{x\in \mathbb{R}}(t,x)=0,~~\forall t>0,\\
        u(0,.)=u^0(.),
    \end{cases}
\end{equation}
where $I$ is a function of bounded variation, given as the almost everywhere limit of $(I^\varepsilon)_\varepsilon$.\\
The limit density $n$ satisfies
$\mathrm{supp}\, n \subset \{ u = 0 \}$. Therefore, a strategy to prove that $n$ concentrates around Dirac masses is to show that the set $\{ u = 0 \}$ consists of a finite number of points. The solution $(u,I)$ to the problem above has been proven to be unique in \cite{CL} (see also \cite{SJM} for a uniqueness result for a closely related model).\\

The analysis of \cite{BMP} is conducted under a mutation kernel with finite exponential moments of any order, which implies that the Hamiltonian 
\begin{equation}
\label{eq:Hamiltonian}
H(x,q)=-R(x,I)-p(x)\int_{\R}G(y)e^{q.y}dy,
\end{equation}
is finite for any values of $q$. Similar results have been  obtained in models where mutations are modeled by a diffusion term (see \cite{BMP,PB}).\\

However, when we consider a kernel which  does not decay faster than exponential, Equation \eqref{eq:constrained HJ} may not be well-defined. The reason is that the Hamiltonian can take infinite values. The analysis changes drastically, since when the kernel does not decay faster than exponentially, the standard regularity theory normally used for such models does not apply. \\
Few works have addressed kernels with slow decay. The case of mutations represented by a fractional Laplacian was studied in \cite{Sepideh} (see also \cite{MM} where the relevant scaling was introduced).  In \cite{BGHP} some other fat-tailed kernels have been considered and relevant scalings have been introduced. However, convergence has only been proved under strong assumptions that keep the solution in the domain of the Hamiltonian, making the analysis much simpler.\\

Stochastic methods, based on individual-based models, have also been developed to study the evolutionary dynamics of populations (see, for instance, \cite{CFM,2CFM,CMT}). These stochastic works lead in particular to the justification of deterministic models \cite{BMP}. Hamilton-Jacobi equations presented above have been usually derived from individual-based models in two steps: the first one is to derive deterministic integro-differential equations such as \eqref{eq:selction-muation} from individual-based models  in the limit of a large population \cite{CFM,2CFM}, the second step is to derive Hamilton-Jacobi equations from integro-differential models in the regime of a small mutation variance as explained above. Recently, a direct derivation of the Hamilton-Jacobi equation from an individual-based model has been obtained in a limit combining large population and small mutations in \cite{CMMT}. This approach relies on a discretization of the trait space with a small parameter that tends to zero as the carrying capacity tends to infinity. To perform the Hopf-Cole transformation of the phenotypic density in a stochastic finite population, it is indeed more convenient to consider a discrete model.\\

In this work, we address simultaneously two aspects: the main one is the extension of the class of mutation kernels to those with only exponential decay, the second one concerns the treatment of the discrete version of   \eqref{eq:scaling selction-muation}, motivated by the study of a stochastic individual-based model in the same spirit as in \cite{CMMT}. The deterministic model that we study in this article corresponds indeed to the limit of large population of the corresponding individual-based model. The current work is therefore a first step to the asymptotic analysis of such an individual-based model. Note also that one can use the same method presented in this article, with less technical difficulties, to study the corresponding continuous model.\\

The main difficulty in this work is that for an exponentially decaying mutation kernel $G$ the Hamiltonian \eqref{eq:Hamiltonian} can take infinite values for some values of $q$. Another difficulty arises from the fact that the population density is known only on a grid, hence introducing difficulties in the analysis of the convergence to a viscosity solution in a continuous trait space. To prove the convergence of the WKB transformation of the density, we combine the method of semi-relaxed limits (see \cite{G,GB-srl}) with Lipschitz estimates in space. The methods based on semi-relaxed limits and Lipschitz estimates are often used alternatively and in classical works one of the methods is usually sufficient. Note that Lipschitz estimates in space are not sufficient here, because we lack regularity in time due to the non finite nature of the Hamiltonian. Here we need to combine the two methods: we use the semi-relaxed limits in a non-classical viscosity procedure, to overcome the difficulty coming from the Hamiltonian with infinite values, and the Lipschitz estimates in $x$ to deal with the discrete aspect of the problem. Note also that here we apply the method of semi-relaxed limits in a rather non-standard manner, closely related to \cite{Sepideh,Chasseine} which also deal with Hamiltonians with possible infinite values. The article \cite{Sepideh} studies a closely related continuous model where the mutations are represented by a fractional Laplace term, and \cite{Chasseine} studies some large deviation type estimates for some non-local models.\\

From a numerical point of view, our work can be related to the discretization of Hamilton–Jacobi equations. In addition to the space discretization studied here, it would be possible to introduce a time discretization, leading to an original numerical scheme whose discretization step depends on the mutation parameter and becomes simultaneously small as the mutation size tends to zero.
This scheme should converge to the corresponding Hamilton–Jacobi equation in a limit that combines both the time and space discretization steps and the mutation parameter.
Note that for models with Gaussian kernels or Laplace mutation operator, asymptotic-preserving numerical schemes have already been proposed in \cite{CHY,GH}, where the discretization step does not need to vanish as the mutation parameter tends to zero.
The analysis developed in this paper provides a basis for extending the analysis in \cite{CHY,GH}, to the study of such numerical schemes in the case of kernels with slow decay.\\

In section 2, we describe the  discrete model, the parameter scaling, and we state our main result. The remaining of this paper is dedicated to the proof of this result, we refer to the end of section 2 for a detailed plan of the proof.
\section{Model and main result}
\subsection{The model}
 We consider a discretization parameter $\delta_\varepsilon \rightarrow 0$ as $\varepsilon \rightarrow 0$. The trait space with discretization $\delta_\K$ is given by $\mathcal{X}_\varepsilon := \{i\delta_\varepsilon,~i \in \mathbb{Z}\}$. We consider a discrete version $(n^\K_i(t))_{i\in \Z}$  of the rescaled selection-mutation model \eqref{eq:scaling selction-muation} in $\mathcal{X}_\varepsilon,$  where $n^\K_i(t)$ is the population density at time $t$ and trait $i\delta_\K,$
 we obtain the following rescaled problem: $\forall (t,i)\in (0,+\infty)\times \mathbb{Z},$
\begin{equation}
\label{eq:rescaling model}
    \begin{cases}
       \varepsilon \frac{d}{dt}n^\K_i(t)=R(i\delta_\K,I^\K(t))n^\K_i(t)+\sum_{l\in \mathbb{Z}}p((l+i)\delta_\K)\frac{\delta_\K}{\K}  G(lh_\K)n^\K_{l+i}(t),\\
       I^\K(t)=\sum_{i\in \mathbb{Z}}\delta_\K n^\K_i(t),\\
        n^\K_i(0)=n^{\K,0}(i\delta_\K),
    \end{cases}
\end{equation}
where $n^{\varepsilon,0}$ is the initial population density, which is assumed to vanish for large traits as the mutation parameter tends to zero. Moreover,
we assume that \begin{equation*}
     \delta_\K\ll \K,
 \end{equation*} 
 so that the trait discretization step allows us to observe mutations with sizes of the order of $\K.$  
 Then, $h_\K:=\frac{\delta_\K}{\K}\rightarrow 0.$\\
Following earlier works \cite{BMP,PB,CMMT}, we expect a concentration phenomenon in the form of Dirac masses. To analyze the density $n^\varepsilon,$ we perform the classical Hopf-Cole transformation 
\begin{equation*}
\label{Hop-Col}
u_i^{\K}(t)=\varepsilon \log{n_i^\K(t)},~~~~~~n^\varepsilon_i(t)=e^{\frac{u^\varepsilon_i(t)}{\varepsilon}}.
\end{equation*}
We obtain  the following  system of ordinary differential equations:
\begin{equation}
\label{eq:system of ODE}
\begin{cases}
  \frac{d}{dt}u^{\K}_i(t)=R(i\delta_\K,I^\K(t))+\sum_{l\in\mathbb{Z}}p((l+i)\delta_\K)h_\K G(lh_\K)e^{\frac{1}{\K}(u^{\K}_{l+i}(t)-u^{\K}_i(t))},~\forall(t,i)\in (0,+\infty)\times\mathbb{Z}, \\
  I^\K(t)=\sum_{i\in \mathbb{Z}}\delta_\K n^\K_i(t),~~~\forall t>0,\\
  u^\K_i(0)=u^{\K,0}(i\delta_\K).
  \end{cases}
\end{equation}

Inspired by the continuous setting, one expects that in the limit $\K\rightarrow 0$, a Hamilton-Jacobi equation, set on a continuous trait space, will be obtained. Following \cite{CMMT}, we define, the affine continuous interpolation of $u^\K_i,$ on the non-discretized trait space $\R$. For all $x\in \mathbb{R},$ let $i\in \mathbb{Z}$ be such that $x\in [i\delta_\K,(i+1)\delta_\K),$ we define 
\begin{equation}
\label{eq:interpolation}
    \widetilde{u}^\K(t,x)=u^\K_i(t)(1-\frac{x}{\delta_\K}+i)+u^\K_{i+1}(t)(\frac{x}{\delta_\K}-i).
\end{equation}
Our goal is to study the convergence of the function $\widetilde{u}^\K$ as $\K\rightarrow 0.$
\subsection{Assumptions}
   \begin{enumerate}
        \item \label{item:1} There exist two positive constants $I_m$ and $I_M$ such that 
        \begin{equation}
            \label{eq:growth bound }
            \inf_{x\in \R} \big(R(x,I_m)+p(x)\big)=0,~~~~~~\sup_{x\in \R}\big(R(x,I_M)+p(x)\big)=0,
        \end{equation}
        and \begin{equation}
        \label{eq:intitial size population bound}
            I_m/2\leq I^\K(0)\leq 2I_M,~~\forall \K>0.
        \end{equation}
       \item \label{item:2} We assume that $R$ and $p$  are Lipschitz-continuous functions, in space with Lipschitz norms $\Vert R \Vert_{\text{Lip}}$ and $\Vert p \Vert_{\text{Lip}}$ uniformly w.r.t $I\in [I_m/2,2I_M]$, and that there exist positive constants $\overline{R},$ $\overline{p},$ and $\underline{p},$ such that for any $x\in \R,$
       \begin{equation}
       \label{growth canditions}
  \sup_{I\in [I_m/2,2I_M]}\Vert R(.,I)\Vert_{W^{2,\infty}(\R)} \leq \overline{R} \quad \text{and} \quad 0 < \underline{p} \leq p(x)\leq \Vert p\Vert_{W^{2,\infty}(\R)} \leq   \overline{p}.
       \end{equation}
      \item \label{item:3} The growth rate $R(.,.)$ is differentiable w.r.t $I$ for all $x\in \R,$ and
      there exist positive constant $A_1, A_2$ and $A_3$ such that for any $I,I_1,I_2\in \R^+$ and $x\in \R$
      \begin{equation}
      \label{eq:growth monotonie in I}
        \vert R(x,I_1)-R(x,I_2)\vert \leq A_1\vert I_1-I_2\vert,~~~-A_2\leq \partial_IR(x,I)\leq -A_3.
 \end{equation}
        
        \item \label{item:4} The kernel $G$ has an exponential decay, i.e., $G(x)=f(x)e^{-\vert x\vert },$  such that  $\int_{\mathbb{R}}G(y)dy=1,$
         $G$ is a positive continuous function in $\R,$ and for all $x\in \R$ \begin{equation} f(x)=f(-x),~~~~~
            0<\min\limits_{x\in \R}f(x)\leq f(x)\leq \sup_{x\in \R}f(x)<+\infty,
            \end{equation}
  which implies that 
     \begin{equation}
           \label{kernel conditions} \int_{\mathbb{R}}G(y)e^{ay}dy<+\infty~~\forall \vert a\vert<1,~\text{and}~~\int_{\mathbb{R}}G(y)e^{ay}dy=+\infty ~\forall \vert a\vert \geq 1.
        \end{equation}
            
   \item \label{item:5} There exist  positive constants $A$ and $B_1$ such that, for any $i\in \mathbb{Z}$ and $\K>0$
    \begin{equation}
    \label{eq:sublinear growth initial conditon}
        u^{\K,0}(i\delta_\K)\leq -A\vert i\delta_\K\vert +B_1.
    \end{equation}
   \item \label{item:6} There exists a positive constant  $L<1$ such that,  for any $i\in \mathbb{Z}$  and $\K>0,$
    \begin{align}
    \label{eq:Lipshitz conditon}
        \left\vert \frac{u^{\K,0}((i+1)\delta_\K)-u^{\K,0}(i\delta_\K)}{\delta_\K}\right \vert \leq L,
        \end{align}
        i.e. for any $\K> 0,$ $u^{\K,0}$ is $L$-Lipschitz continuous uniformly in $\varepsilon,$ in the discrete trait space $\mathcal{X}_\K$.
    \item \label{item:7} The linear interpolation $\widetilde{u}^{\K,0}$ of $u^{\K,0},$ defined as in \eqref{eq:interpolation}, converges locally uniformly to a continuous function $u^0,$ as $\K\rightarrow 0$.
\end{enumerate}
From Assumption \ref{item:1}, the population evolves under limited resources: when its total size is large, the growth rate is negative for all traits, whereas when the total size falls below a minimal threshold, the growth rate is positive for all traits. The second part of Assumption \ref{item:3} models competition. As the population size increases, the growth rate decreases for all traits.  Assumption \ref{item:4} is the new assumption in this setting. The mutation kernel is assumed to have only exponential decay, instead of the super exponential decay considered in the literature.  An example of Kernel $G$ satisfying Assumption \ref{item:4} is given by the double exponential kernel $G(x)=\frac{1}{2}e^{-\vert x\vert }.$ The other assumptions are technical.


\subsection{Main result}
This section is devoted to the statement of our main result.
\begin{thm}
    Under Assumptions \ref{item:1}-\ref{item:7}, let $u^\K=(u^\K_i)_{i\in\mathbb{Z}}$ be the solution of \eqref{eq:system of ODE} and $\widetilde{u}^\K$ as in \eqref{eq:interpolation}. Then, as $\K\rightarrow 0,$ along a subsequence, $(I^\varepsilon)_\varepsilon$ converges almost everywhere to a function of bounded variation $I,$ and $\widetilde{u}^\K$ converges locally uniformly to a function $u\in C([0,+\infty)\times \mathbb{R}),$  viscosity solution to the following equation 
    \begin{equation}
    \label{eq:HJ}
    \begin{cases}
      \min\Big( \partial_tu(t,x) -R(x,I(t))-p(x)\int_{\mathbb{R}}G(y)e^{\nabla u(t,x).y} dy, 1-\vert \nabla u(t,x)\vert \Big)=0,~~\forall (t,x)\in(0,+\infty)\times  \mathbb{R},\\
       \max_{x\in \mathbb{R}}u(t,x)=0,~~\forall t>0,\\
      u(0,.)=u^0(.).
      \end{cases}
    \end{equation}
    Moreover, along a subsequence, the linear interpolation  of $(n^\K_i)_{i\in \Z},$ denoted by
 $\widetilde{n}^\K,$ converges in $L^\infty(w*(0,+\infty),\mathcal{M}^1(\mathbb{R}))$ to a measure $n$ which satisfies  for almost every $t>0,$  
    \begin{equation}
    \label{eq:concentration}
    \text{supp}~n(t,.)\subset \{x\in \mathbb{R}/ u(t,x)=0\}.
    \end{equation}
 \end{thm}
To the best of our knowledge, such a Hamilton-Jacobi equations with obstacle is new in the field of models from evolutionary biology. It results from the exponential decay of the kernel. Equation \eqref{eq:HJ} has the form of the classical Hamilton-Jacobi equation  $$\partial_t u(t,x) -R(x,I(t))-p(x)\int_{\mathbb{R}}G(y)e^{\nabla u(t,x).y} dy=0,$$  when $\vert \nabla u \vert<1.$ The fact that, the Hamiltonian takes infinite values when $\vert \nabla u\vert \geq 1$,  enforces $u$ to satisfy $\vert \nabla u\vert \leq 1.$\\ 

 We expect that \eqref{eq:HJ} admits a unique solution $(u,I)$ with $u\in C(\R^+\times \R)$ solving \eqref{eq:HJ} in the viscosity sense and $I\in BV_{Loc}(\R^+)$. Such a uniqueness result has been proved for constrained Hamilton-Jacobi equations without obstacle, derived from models with fast decaying mutation kernels \cite{CL}. The extension of such  uniqueness results to \eqref{eq:HJ}, requires some work since here  the Hamiltonian can take infinite values and the problem involves an obstacle of type $|\nabla u|\leq 1$. This goes beyond  the scope of this article.\\
 However, a uniqueness result for fixed $I$ can be obtained using classical arguments combining the doubling variables method (see for instance \cite{G}) with the notion of viscosity solutions for $L^1$ time dependence Hamiltonians (see \cite{Ishi}), following similar techniques used in the proof of Proposition 5.4.\\
 
The remaining of this paper is dedicated to the proof of Theorem 2.1: In Section 3, we present results regarding the existence, uniqueness, and comparison principle for the equation \eqref{eq:rescaling model}. In Section 4, we establish the local bounds and Lipschitz estimates of the solutions \eqref{eq:system of ODE}, which are crucial ingredients in proving convergence. Sections 5 and 6 are dedicated to the proof of Theorem 2.1.
\section{Preliminary results and definitions}
In this section, we give some preliminary results concerning the discrete equation \eqref{eq:rescaling model}, and the definition of the viscosity solution for a discontinuous Hamiltonian.  
\begin{lem}
  For any $a\in(-1,1),$ the sum $\sum_{l\in \Z}h_\K G(l h_\K )e^{a \vert lh_\K\vert  }$ converges to $\int_{\R}G(y)e^{a\vert y \vert }dy$ as $\K\rightarrow 0,$ and there exists a finite constant $\alpha(a)$ such that 
  \begin{equation}
      \sup_{\K>0} \sum_{l\in \Z}h_\K G(lh_\K)e^{a\vert lh_\K \vert }=:\alpha(a).
  \end{equation}
\end{lem}
\begin{proof}
    Let $a\in(-1,1),$ and $M>0.$ By continuity of $G,$ the Riemann sum $\sum_{\vert lh_\K\vert \leq M}h_\K G(l h_\K)e^{a \vert lh_\K \vert }$ converges to $\int_{\vert x\vert\leq  M}G(x)e^{a\vert x \vert}dx.$
Moreover, by Assumption \ref{item:4}, we have 
\begin{align*}
    \sum_{ \vert lh_\K\vert >M }h_\K G(l h_\K)e^{a  \vert lh_\K \vert }&\leq C \sum_{ \vert lh_\K\vert >M}h_\K e^{-\vert lh_\K\vert} e^{a  \vert lh_\K \vert}\\
    & \leq C \int_{\vert y\vert \geq M}e^{-(1-a)\vert y\vert}dy.
\end{align*}
Then 
$$\limsup\limits_{\K\rightarrow 0}   \sum_{ \vert l h_\K \vert >M}h_\K G(l h_\K)e^{a  \vert l h_\K \vert}\leq C \int_{\vert y\vert \geq M}e^{-(1-a)\vert y \vert}dy. $$
Letting $M$ tend to infinity, we deduce the convergence of the entire sum $\sum_{l\in \Z}h_\K G(l h_\K)e^{a \vert lh_\K \vert }$ to $\int_{\R}G(y)e^{a\vert y\vert}dy,$ as $\K\rightarrow 0.$

\end{proof}
\begin{thm}
 Under Assumptions \ref{item:1}-\ref{item:5}, for all $\K>0,$ there exists a unique solution \\$n^\K\in C^1(\R^+,\ell^1(\mathbb{Z}))$ to Equation \eqref{eq:rescaling model} which satisfies, for all $t\geq 0,$ 
\begin{equation}
\label{eq:bound of totale size}
   I_m/2\leq I^\K(t)\leq 2I_M.
\end{equation}  
\end{thm}
See Appendix A for the proof.
\begin{pro}
 The sequence of functions  $(I^\K)_\K$ is locally uniformly of bounded variation, and we have for all $T>0,$
  \begin{equation}
    \int_{0}^{T}  \left| \frac{dI^\K(t)}{dt} \right| dt\leq  C(T),
  \end{equation}
  where $C(T)$ is a positive constant depending on $T$ but not on $\K$. Then, along subsequences, $(I^\K)_\K$ converges almost everywhere to a function $I,$ which is locally of bounded variation and non-decreasing in $(0,+\infty)$. 
\end{pro}
The proof is given in Appendix A. 
\begin{pro}
   For any fixed $\K>0$ and $I^\K\in \R^+,$ Equation \eqref{eq:rescaling model} satisfies the comparison principle for the class of functions $u \in C^1(\R^+,\ell^1(\mathbb{Z})),$ such that 
   \begin{equation}
   \label{eq:localy bound condition}
      \sup_{t\in[0,T]}\sum_{i\in \Z}\vert u_i(t)\vert<+\infty,~~\forall T>0.
        \end{equation}
        By this we mean that, if $v \in C^1(\R^+,\ell^1(\mathbb{Z}))$ is a subsolution of \eqref{eq:rescaling model} satisfying \eqref{eq:localy bound condition}, and $w\in C^1(\R^+,\ell^1(\mathbb{Z}))$ is a supersolution of \eqref{eq:rescaling model} satisfying \eqref{eq:localy bound condition},
   such that $v(0,.)\leq w(0,.)$, then we have  $$v(t,.)\leq w(t,.),~~~ \forall t\geq 0.$$
\end{pro}
\begin{proof}
In this proof, we omit the fixed index $\K.$ \\
Since Equation \eqref{eq:rescaling model} is linear for fixed $I^\K$, it suffices to show that for any subsolution  $v\in C^1(\R^+,\ell^1(\mathbb{Z}))$ satisfying \eqref{eq:localy bound condition}, such that $v(0, \cdot) \leq 0$, then  $v(t, \cdot) \leq 0,$~ $\forall t \geq 0$.\\
    Let $v\in C^1(\R^+,\ell^1(\mathbb{Z}))$  be a subsolution of \eqref{eq:rescaling model} satisfying \eqref{eq:localy bound condition} such that $v(0, \cdot) \leq 0.$\\ Define, for $(t,i)\in [0,+\infty)\times \mathbb{Z}$ and for $D$ to be chosen later $\widetilde{v}_i(t):=v_i(t)e^{-\frac{1}{2}|ih_\K\vert } e^{-Dt}.$   
    The function $\widetilde{v}$ satisfies the following inequality 
    \begin{equation}
    \label{eq:vanish equation}
        \varepsilon\frac{d}{dt}\widetilde{v}_i(t)\leq(-D+R(i\delta_\K,I^\K(t)))\widetilde{v}_i(t)+e^{-\frac{1}{2}|ih_\K\vert }\sum_{l\in \mathbb{Z}}p((l+i)\delta_\K)h_\K G(lh_\K)e^{\frac{1}{2}\vert (l+i)h_\K\vert  }\widetilde{v}_{l+i}(t).
    \end{equation}
    Let $T>0,$ we show that $\widetilde{v}_i(t)\leq 0,$~ $\forall t\in[0,T].$
   Let us assume by contradiction that  $$M=\sup_{(t,i)\in [0,T]\times \mathbb{Z}}\widetilde{v}_i(t)>0.$$
   By \eqref{eq:localy bound condition}, we deduce that  $\widetilde{v}$  vanishes when $\vert i\vert \rightarrow +\infty $. Thus, the supremum is attained at some point  $(t_0,i_0).$ Since $\widetilde{v}(0,.)\leq 0,$ we conclude that $t_0>0.$ Consequently, we obtain
   \begin{align*}
      \varepsilon\frac{d}{dt}\widetilde{v}_{i_0}(t_0)&\leq (-D+R(i_0\delta_\K,I^\K(t_0)))M+M\sum_{l\in \mathbb{Z}}p((l+i_0)\delta_\K)h_\K G(lh_\K)e^{\frac{1}{2}\vert lh_\K \vert  }\\
   &\leq (-D+\overline{R}+\overline{p}\alpha(1/2))M.
\end{align*}
 We choose $D$ sufficiently large  such that $-D+\overline{R}+\overline{p}\alpha(1/2)<0.$ This is in contradiction with  $t_0>0.$ 
\end{proof}
Here, we recall the notion of viscosity solution for a discontinuous Hamiltonian (see \cite{G} page 80). Note that the possibility of discontinuity of the Hamiltonian comes from the fact that the function $I$ of Proposition 3.3 is only of bounded variation, and is potentially discontinuous.
\begin{df}
Let $I:[0,+\infty)\rightarrow \R^+$ be a function of bounded variation.
    \begin{itemize}
        \item An upper semi-continuous function $u$ is a viscosity subsolution of \eqref{eq:HJ} if only if, for all $\phi \in C^\infty(\R^+\times \mathbb{R}),$ and $(t,x)\in\R^+\times \mathbb{R} ,$ such that $u-\phi$ has a maximum at $(t,x)$, we have
\begin{equation}
    \partial_t\phi(t,x)\leq \limsup\limits_{s\rightarrow t}R(x,I(s))+p(x)\int_{\R}G(y)e^{\nabla \phi(t,x).y}dy, ~~\text{or}~~~\vert \nabla \phi(t,x)\vert\geq 1.
\end{equation}
\item A lower semi-continuous function $v$ is a viscosity supersolution of \eqref{eq:HJ} if only if, for all $\phi \in C^\infty(\R^+\times \mathbb{R}),$ and $(t,x)\in \R^+\times \mathbb{R} ,$ such that $u-\phi$ has a minimum at $(t,x),$ we have
\begin{equation}
    \partial_t\phi(t,x)\geq \liminf\limits_{s\rightarrow t}R(x,I(s))+p(x)\int_{\R}G(y)e^{\nabla \phi(t,x).y}dy,~~\text{and}~~~\vert \nabla \phi(t,x)\vert\leq 1.
\end{equation}
    \end{itemize}
\end{df}
Note that from Assumption \ref{item:3}, we deduce that $$\limsup\limits_{s\rightarrow t}R(x,I(s))=R(x,\liminf\limits_{s\rightarrow t}I(s))=:R(x,\underline{I}(t)),$$ and $$\liminf\limits_{s\rightarrow t}R(x,I(s))=R(x,\limsup\limits_{s\rightarrow t}I(s))=:R(x,\overline{I}(t)).$$ 
\section{Regularity estimates}
In this section, we prove  linear growth bounds and spatial Lipschitz estimates  for $u^\varepsilon$. 
\begin{lem}
    Let $\K>0.$ Under Assumptions \ref{item:1}-\ref{item:6}, we have for any $i\in \mathbb{Z},$ and $ t\geq 0$  
    \begin{equation}
    \label{eq:semi lineair growth}
      -L|i\delta_\K|+
        B_2+C_2t\leq u^{\K,0}(i\delta_\K)+C_2t  \leq u^\K_i(t)\leq u^{\K,0}(i\delta_\K)+C_1t\leq -A\vert i \delta_\K\vert +B_1+C_1t,
         \end{equation}
    where $C_2,B_2$ and $C_1$ are real constants. 
\end{lem}
This lemma establishes the decay of $u^\varepsilon$
 at infinity, a property propagated from the initial condition.
\begin{proof}
  The proof of the upper bound is a direct application of the comparison principle proved in Proposition 3.4. Let us show that $\phi_i(t):=u^{\K,0}(i\delta_\K)+C_1t,$ is a supersolution of the equation \eqref{eq:system of ODE} with the same  $I^\K$.
        Indeed, we have 
        $$\phi_i'(t)=C_1~~\text{and}
        ~~ \phi_{l+i}(t)-\phi_i(t)\leq L \vert l\delta_\K\vert .$$
     We can choose $C_1=\overline{R}+\overline{p}\alpha(L).$ Moreover, from \eqref{eq:sublinear growth initial conditon}, we conclude that  $(t,i)\mapsto e^{\frac{1}{\K}\phi_i(t)}$ is a supersolution  of \eqref{eq:rescaling model} satisfying \eqref{eq:localy bound condition}. By comparison principle, we deduce that $$n_i^\K(t)\leq e^{\frac{1}{\K}\phi_i(t)},~~\forall t\geq 0.$$
     Then from \eqref{eq:sublinear growth initial conditon}, we conclude that $$u^\K_i(t)\leq \phi_i(t):=u^{\K,0}(i\delta_\K)+C_1t\leq-A\vert i \delta_\K\vert+B_1+C_1t,~~\forall t\geq 0.$$
 Let us now show the lower bound. From \eqref{eq:system of ODE}, and by using \eqref{eq:bound of totale size} and Assumption \ref{item:2} we obtain
  
    $$\frac{d}{dt}u^\K_i(t)\geq - \overline{R}.$$
Then, by \eqref{eq:Lipshitz conditon}, we deduce  that \begin{align*}
    u^\K_i(t)\geq- \overline{R}t + u^{\K,0}(i\delta_\K)\geq -\overline{R}t+\inf_{\K>0} u^{\K,0}(0)-L\vert i\delta_\K\vert. 
\end{align*}
\end{proof}
The next proposition gives a Lipschitz estimate in space, which is an important ingredient in the proof of Theorem 2.1.
\begin{pro}
    Let $T>0,$ and $K\geq 0.$ Under Assumptions \ref{item:1}-\ref{item:6}, there exists a constant $C$ which may depend on $T$ and $K$ such that for any $\K>0,$ $t \in [0,T]$ and $i\in \mathbb{Z}$ such that $i\delta_\K\in [-K,K],$ we have 
    \begin{equation}
   \vert \Delta_\K u^{\K}_i(t)\vert \leq C,
    \end{equation}
    where  $\Delta_\K u^{\K}_i(t)=\frac{u^\K_{i+1}(t)-u^\K_{i}(t)}{\delta_\K}.$ 
\end{pro}

\begin{proof} 
The proof is inspired by Bernstein’s method, which is usually used in the continuous setting. The idea is to establish an equation satisfied by the discrete derivative $\Delta_\varepsilon$ and then conclude by applying a maximum principle. However, this approach cannot be used here due to the space dependence on the mutation rate, which prevents the propagation of the Lipschitz initial bound (see \cite{Bl} for more details about the regularity of integral equations). To handle this issue, we consider a function of the form $\Delta_\varepsilon u^\varepsilon + C u^\varepsilon,$ to which we then apply the maximum principle.

The calculations in this proof are an adaptation of those in \cite{CMMT,BMP}.  In this proof, $C$ is a constant that may change from line to line.\\
Let $T>0.$ For any  $t \in [0,T]$ and $i\in \mathbb{Z},$  we have  

\begin{align*}
  \frac{d}{dt} \Delta_\K u^{\K}_i(t)=&\frac{ R((i+1)\delta_\K,I^\K(t))-R(i\delta_\K,I^\K(t))}{\delta_\K}+\frac{1}{\delta_\K}\Bigg[\sum_{l\in \mathbb{Z}}p((l+i+1)\delta_\K)h_\K\\
  &~~~~~~~~~~~~~~~G(lh_\K)e^{\frac{1}{\K}(u^\K_{l+i+1}(t)-u^{\K}_{i+1}(t))}-\sum_{l\in \mathbb{Z}}p((l+i)\delta_\K)h_\K G(lh_\K)e^{\frac{1}{\K}(u^\K_{l+i}(t)-u^{\K}_{i}(t))}\Bigg].
\end{align*}
Moreover, we have
\begin{equation}
\label{eq:mutaion inqualitie}
    p((l+i+1)\delta_\K)-p((l+i)\delta_\K)\leq \Vert p\Vert_{Lip}\delta_\K\leq \frac{\Vert p\Vert_{Lip}}{\underline{p}}\delta_\K p((l+i)\delta _\K).
    \end{equation}

Then 
\begin{align*}
     &\frac{1}{\delta_\K}\left[\sum_{l\in \mathbb{Z}}p((l+i+1)\delta_\K)h_\K G(lh_\K)e^{\frac{1}{\K}(u^\K_{l+i+1}(t)-u^{\K}_{i+1}(t))}-\sum_{l\in \mathbb{Z}}p((l+i)\delta_\K)h_\K G(lh_\K)e^{\frac{1}{\K}(u^\K_{l+i}(t)-u^{\K}_{i}(t))}\right]\\
     & \leq \frac{\Vert p\Vert_{Lip}}{\underline{p}}\sum_{l\in \mathbb{Z}}p((l+i)\delta_\K)h_\K G(lh_\K)e^{\frac{1}{\K}(u^\K_{l+i+1}(t)-u^{\K}_{i+1}(t))}\\
      &~~~~~~~~~~~~~~~~~~~~~~~~~~~+\frac{1}{\delta_\K}\sum_{l\in \mathbb{Z}}p((l+i)\delta_\K)h_\K G(lh_\K)\Big[e^{\frac{1}{\K}(u^\K_{l+i+1}(t)-u^{\K}_{i+1}(t))}-e^{\frac{1}{\K}(u^\K_{l+i}(t)-u^{\K}_{i}(t))}\Big]\\
  &\leq \frac{\Vert p\Vert_{Lip}}{\underline{p}}\sum_{l\in \mathbb{Z}}p((l+i)\delta_\K)h_\K G(lh_\K)e^{\frac{1}{\K}(u^\K_{l+i+1}(t)-u^{\K}_{i+1}(t))}\\
  &~~~~~~~~~~~~~~~~~~~~~~~~+\frac{1}{\K}\sum_{l\in \mathbb{Z}}p((l+i)\delta_\K)h_\K G(lh_\K)e^{\frac{1}{\K}(u^\K_{l+i+1}(t)-u^{\K}_{i+1}(t))}\big(\Delta_\K u^\K_{l+i}(t)-\Delta_\K u^{\K}_{i}(t)\big).
  \end{align*}
In the last line we have used the inequality  $e^x-e^y\leq e^x(x-y)$ for all $x,y\in \mathbb{R}.$\\
 We obtain that
\begin{align*} 
 \frac{d}{dt}\Delta_\K u^\K_i(t) &\leq \Vert R\Vert_{Lip}+ \frac{\Vert p\Vert_{Lip}}{\underline{p}}\sum_{l\in \mathbb{Z}}p((l+i)\delta_\K)h_\K G(lh_\K)e^{\frac{1}{\K}(u^\K_{l+i+1}(t)-u^{\K}_{i+1}(t))}\\
  &~~~~~~~+\frac{1}{\K}\sum_{l\in \mathbb{Z}}p((l+i)\delta_\K)h_\K G(lh_\K)e^{\frac{1}{\K}(u^\K_{l+i+1}(t)-u^{\K}_{i+1}(t))}\big(\Delta_\K u^\K_{l+i}(t)-\Delta_\K u^{\K}_{i}(t)\big)\\  
  &\leq \Vert R\Vert_{Lip}+\frac{\Vert p\Vert_{Lip}}{\underline{p}}\sum_{l\in \mathbb{Z}}p((l+i)\delta_\K)h_\K G(lh_\K)e^{\frac{1}{\K}(u^\K_{l+i+1}(t)-u^{\K}_{i+1}(t))}\\
&~~~~~~~~~~~~~~~~~~~~~~~~~~~~~~~~~~~~~~~~~~~~~~~~~~~~~~~~~~~~~~~~~~~~\times\left(1-\frac{1}{\K}(u^\K_{l+i+1}(t)-u^{\K}_{i+1}(t))\right)\\
  &~~~~~~~~~~+\frac{1}{\K}\sum_{l\in \mathbb{Z}}p((l+i)\delta_\K)h_\K G(lh_\K)e^{\frac{1}{\K}(u^\K_{l+i+1}(t)-u^{\K}_{i+1}(t))}\\
   &~~~~~~~~~~~~~~~~~~~~~~~~~~~\times \left(\Delta_\K u^\K_{l+i}(t)+\frac{\Vert p\Vert_{Lip}}{\underline{p}}u^\K_{l+i+1}(t)-\big(\Delta_\K u^{\K}_{i}(t)+\frac{\Vert p\Vert_{Lip}}{\underline{p}} u^\K_{i+1}(t)\big)\right).
\end{align*}
On the other hand, from \eqref{eq:mutaion inqualitie} we obtain for $\K$  small enough $$p((i+l+1)\delta_\K) \leq 2 p((l+i)\delta_\K).$$ 
 We deduce from \eqref{eq:system of ODE} that for $\K$ small enough,
\begin{align*}
    \frac{d}{dt}u^\K_{i+1}(t) &\leq \overline{R}+\sum_{l\in \mathbb{Z}}p((l+i+1)\delta_\K)h_\K G(lh_\K)e^{\frac{1}{\K}(u^\K_{l+i+1}(t)-u^{\K}_{i+1}(t))}\\
    &\leq \overline{R}+2\sum_{l\in \mathbb{Z}}p((l+i)\delta_\K)h_\K G(lh_\K)e^{\frac{1}{\K}(u^\K_{l+i+1}(t)-u^{\K}_{i+1}(t))}.
\end{align*}
From these inequalities, we obtain 
\begin{align*}
    \frac{d}{dt} &\left(\Delta_\K u^\K_i(t) +\frac{\Vert p\Vert_{Lip}}{\underline{p}}u^\K_{i+1}(t)\right) \\
    &\leq C + \frac{\Vert p\Vert_{Lip}}{\underline{p}}  \sum_{l\in \mathbb{Z}} p((l+i)\delta_\K)h_\K G(lh_\K)e^{\frac{1}{\K}(u^\K_{l+i+1}(t)-u^{\K}_{i+1}(t))}\times \big(3-\frac{1}{\K}(u^\K_{l+i+1}(t)-u^\K_{i+1}(t))\big)\\
    &~~+ \frac{1}{\K} \sum_{l\in \mathbb{Z}} p((l+i)\delta_\K)h_\K G(lh_\K)e^{\frac{1}{\K}(u^\K_{l+i+1}(t)-u^{\K}_{i+1}(t))} \\
    &~~~~~~~~~~~~~~~~~~~~~~~~~~~~~~~~~~~~~~~~~~~~~~\quad \times \left(\Delta_\K u^\K_{l+i}(t) + \frac{\Vert p\Vert_{Lip}}{\underline{p}}u^\K_{l+i+1}(t)-\big(\Delta_\K u^{\K}_{i}(t) + \frac{\Vert p\Vert_{Lip}}{\underline{p}} u^\K_{i+1}(t)\big)\right),
\end{align*}
where $C$ is a constant independent of $\K$. Since $e^x(3-x)\leq e^2 $ for all $x\in \mathbb{R},$ we deduce that 

\begin{align*}
    \frac{d}{dt}\left (\Delta_\K u^\K_i(t) +\frac{\Vert p\Vert_{Lip}}{\underline{p}}u^\K_{i+1}(t)\right)&
    \leq 
    C +\overline{p}\frac{1}{\K}\sum_{l\in \mathbb{Z}}h_\K G(lh_\K)e^{\frac{1}{\K}(u^\K_{l+i+1}(t)-u^\K_{i+1}(t))}\\
     &~~~~~~~~~~\times \left (\Delta_\K u^\K_{l+i}(t)+\frac{\Vert p\Vert_{Lip}}{\underline{p}}u^\K_{l+i+1}(t)-\big(\Delta_\K u^{\K}_{i}(t)+\frac{\Vert p\Vert_{Lip}}{\underline{p}} u^\K_{i+1}(t)\big)\right), 
\end{align*}
where the constant $C$ is independent of $\K.$
Let us define $$g^\K_i (t):=\Delta_\K u^\K_i(t) +\frac{\Vert p\Vert_{Lip}}{\underline{p}}u^\K_{i+1}(t)-2Ct.$$
 We obtain 
\begin{align}
 \frac{d}{dt}g^\K_i (t)\leq -C+\overline{p}\frac{1}{\K}\sum_{l\in \mathbb{Z}}h_\K G(lh_\K)e^{\frac{1}{\K}(u^\K_{l+i+1}(t)-u^\K_{i+1}(t))}(g^\K_{l+i}(t)-g^\K_i(t)).
\end{align}
From Lemma 4.1 and Assumption \eqref{item:6}, we deduce that
$$u^\K_{l+i+1}(t)-u^\K_{i+1}(t)\leq L\vert l\delta_\K\vert+(C_1-C_2)t~~\text{
and}~~g^\K_i(t)\leq L+\frac{(C_1-C_2)t}{\delta_\K}+C(T).$$ Moreover for $\alpha>0,$  we have $g^\K_i-\alpha \vert i\delta_\K\vert\rightarrow-\infty$ as $\vert i\vert\rightarrow +\infty.$ Let $(t_{\alpha,\K},i_{\alpha,\K})$ be a maximum point of $g^\K_i -\alpha\vert i\delta_\K\vert$ on $[0,T]\times \Z.$ By contradiction we assume that $t_{\alpha,\K}>0,$ then 
 \begin{align*}
 \frac{d}{dt}g^\K_{i_{\alpha,\K}} (t_{\alpha,\K})&\leq-C+\overline{p}\frac{1}{\K}\sum_{l\in \mathbb{Z}}h_\K G(lh_\K)e^{\frac{1}{\K}(u^\K_{l+i_{\alpha,\K}+1}(t_{\alpha,\K})-u^\K_{i_{\alpha,\K}+1}(t_{\alpha,\K}))}(g^\K_{l+i_{\alpha,\K}}(t_{\alpha,\K})-g^\K_{i_{\alpha,\K}}(t_{\alpha,\K}))\\
 & \leq -C+\alpha \overline{p}\frac{1}{\K}\sum_{l\in \mathbb{Z}}h_\K G(lh_\K)e^{\frac{1}{\K}(u^\K_{l+i_{\alpha,\K}+1}(t_{\alpha,\K})-u^\K_{i_{\alpha,\K}+1}(t_{\alpha,\K}))}\big(\vert (l+i_{\alpha,\K})\delta_\K)\vert -\vert i_{\alpha,\K}\delta_\K\vert \big)
 \\
 &\leq -C+\alpha \overline{p} \sum_{l\in \mathbb{Z}}h_\K G(lh_\K)\vert lh_\K\vert e^{L\vert lh_\K\vert +\vert C_1-C_2\vert T\frac{1}{\K}}.
\end{align*}
Therefore, for $\alpha$ small enough we have 
$$\frac{d}{dt}g^\K_{i_{\alpha,\K}} (t_{\alpha,\K})<0.$$
This contradicts the fact that $g^\K_{i} -\alpha \vert i\delta_\K \vert$ is maximal at $(t_{\alpha,\K},i_{\alpha,\K})$. We deduce that 
\begin{align}
   \Delta_\K u^\K_i(t) +\frac{\Vert p\Vert_{Lip}}{\underline{p}}u^\K_{i+1}(t)-2Ct-\alpha \vert i\delta_\K\vert \leq \sup\limits_{i\in\mathbb{Z}}  \big(\Delta_\K u^\K_i(0)+\frac{\Vert p\Vert_{Lip}}{\underline{p}} u^\K_{i+1}(0)\big).
\end{align}
Letting  $\alpha$ go to zero and using Lemma 4.1, we conclude that $\Delta_\K u^\K_i$ is locally uniformly bounded above by a constant independent of $\K,$ which may depend on the trait $i\delta_\K,$ but which grows at most linearly in $i\delta_\K$.
In a similar manner, we show that $-\Delta_\K u^\K_i$ is bounded above by a constant independent of $\K$ which grows at most linearly in $i\delta_\K$.
\end{proof}
\section{Convergence to a viscosity solution}
By abuse of notation, throughout the remainder of this paper, $\K$ denotes the subsequence given by Proposition 3.3.\\
We introduce the following semi-relaxed limits (see \cite{G}): for all $t\geq 0$ and $x\in \R,$
\begin{equation*}
    \overline{u}(t,x):=\limsup\limits_{\underset{(s,y)\rightarrow (t,x)}{\K\rightarrow 0}}\widetilde{u}^\K(s,y)~~~\text{and}~~~~ \underline{u}(t,x):=\liminf\limits_{\underset{(s,y)\rightarrow (t,x)}{\K\rightarrow 0}}\widetilde{u}^\K(s,y).
\end{equation*}
Note that we can define such quantities, as $\widetilde{u}^\K$ is locally uniformly bounded, by Lemma 4.1.\\
The method of semi-relaxed limits consists of proving that $\overline{u}$ is a subsolution and $\underline{u}$ is a supersolution, and then concluding by a comparison principle that $\overline{u} = \underline{u}$. Here, this method cannot be used in the classical way for two main reasons. First, the Hamiltonian may take infinite values. Second, our equation is defined only on a discrete grid, while the viscosity notion is pointwise. To overcome the difficulty arising from the discrete nature of the problem, we use spatial Lipschitz estimates to control the difference between the point where we test the viscosity property and its discrete approximation in the grid, for which the equation is given. Then, to deal with the Hamiltonian taking infinite values, we first prove in Proposition 5.1 that $\underline{u}$ is a supersolution. Next, we prove directly that $\overline{u} = \underline{u}$ in Proposition 5.4, and we finally use the properties of $\underline{u}$ to deduce that $\overline{u}$ is also a subsolution in Proposition 5.3.
\subsection{Viscosity supersolution}
\begin{pro}
Under Assumptions \ref{item:1}-\ref{item:7}, the function $\underline{u}$ is a viscosity supersolution for the equation \eqref{eq:HJ}.
\end{pro}
\begin{proof}

\textbf{Step 1: \underline{u} is a viscosity supersolution of \eqref{eq:HJ} in $(0,+\infty)\times \R$}.\\
 Let $\phi \in C^\infty((0,+\infty)\times \mathbb{R}),$ and $(t,x)\in (0,+\infty)\times \R$ be a strict global minimum of  $\underline{u}-\phi$ on $(0,+\infty)\times \R$. Then (see \cite{G}, Lemma 4.2) there exist a subsequence of $\K,$ by abuse of notation still denoted $\K,$  and a sequence $(t_\K,x_\K)$ such that for $\K$ small enough, $\widetilde{u}^\K-\phi$
 has a global minimum at $(t_\K,x_\K)$ and $(t_\K,x_\K)\rightarrow (t,x),$ as $\K\rightarrow 0$  and  $\widetilde{u}^\K(t_\K,x_\K)\rightarrow \underline{u}(t,x).$ In addition, since $(t_\K,x_\K)$ is a global minimum of $\widetilde{u}^\K-\phi,$  we have 
 \begin{equation}
 \label{eq:min and interpolation}
\partial_t \phi(t_\K,x_\K)\geq \partial_t \widetilde{u}^\K(t_\K,x_\K).
\end{equation}
By \eqref{eq:interpolation} we deduce that  
\begin{equation}
 \label{eq: interpolation inequality}
\partial_t \phi(t_\K,x_\K)\geq  \frac{d}{dt}u^\K_{i_\K}(t_\K)(1-\frac{x_\K}{\delta_\K}+i_\K)+\frac{d}{dt}u^\K_{i_\K+1}(t_\K)(\frac{x_\K}{\delta_\K}-i_\K), 
\end{equation}
where $i_\K=\lfloor \frac{x_\K}{\delta_\K}\rfloor$.
 Since $x\mapsto R(x,.)$ is Lipschitz continuous and $I\mapsto R(.,I)$ is decreasing then 

 $$\liminf\limits_{\K\rightarrow 0}R(i\delta_\K,I^\K(t_\K))=R(x,\limsup\limits_{\K\rightarrow 0} I^\K(t_\K)).$$
Moreover from \cite{PB} (proof of Theorem 4.1), we know that
 \begin{equation}
\limsup\limits_{\underset{s\rightarrow t}{\K \rightarrow 0}}I^\K(s)\leq \limsup_{s\rightarrow t}I(s)=:\overline{I}(t).
 \end{equation}
  Thus, $\liminf\limits_{\K\rightarrow 0}R(i\delta_\K,I^\K(t_\K))\geq R(x,\overline{I}(t)).$ We deduce that,  
\begin{equation}
\label{eq:growth lipschitz}
  \liminf\limits_{\K\rightarrow 0} (1-\frac{x_\K}{\delta _\K}+i_\K)R((i_\K+1)\delta_\K,I^\K(t_\K))+(\frac{x_\K}{\delta_\K}-i_\K)R(i_\K\delta_\K,I^\K(t_\K))\geq R(x, \overline{I}(t)).
\end{equation}
Furthermore, for any $l \in \mathbb{Z},$ we have  
\begin{align*}
  u^\K_{l+i_\K}(t_\K)-u^\K_{i_\K}(t_\K) & =u^\K_{l+i_\K}(t_\K)-\widetilde{u}^\K(t_\K,x_\K)+\widetilde{u}^\K(t_\K,x_\K)-u^\K_{i_\K}(t_\K) \\
   & \geq \phi(t_\K,(l+i_\K)\delta_\K))-\phi(t_\K,x_\K)+\widetilde{u}^\K(t_\K,x_\K)-u^\K_{i_\K}(t_\K)\\
   & \geq \phi(t_\K,(l+i_\K)\delta_\K))-\phi(t_\K,x_\K)-C(x)\delta_\K,
\end{align*}
where $C(x)$ is the spatial Lipschitz constant  of $u^\K$ in $[x-1,x+1]$ given by Proposition 4.2.\\
Let $M>0.$ For $\K$ small enough, and for any $\vert l \vert  \leq \lfloor \frac{M}{h_\K}\rfloor,$ there exists a point $\xi_{\K,l}\in[x-1,x+1]$ such that
\begin{align*}
    \phi(t_\K,(l+i_\K)\delta_\K)-\phi(t_\K,x_\K)&=((l+i_\K)\delta_\K-x_\K)\partial_x\phi(t_\K,x_\K)+\frac{1}{2}((l+i_\K)\delta_\K-x_\K)^2\partial_{xx}\phi(t_\K,\xi_{\K,l})\\
    &\geq l\delta_\K\partial_x\phi(t_\K,x_\K)-C\vert x_\K-i_\K\delta_\K\vert -C'(l\delta_\K)^2\\
    &\geq l\delta_\K\partial_x\phi(t_\K,x_\K)-C\delta_\K-C'\varepsilon^2 M^2,
\end{align*}
where $C$ and $C'$ are constants independent of $\K,$ and depending on $\partial_x \phi(t,x),$ and $\partial_{xx} \phi(t,x)$. Therefore,
$$\frac{1}{\K}( u^\K_{l+i_\K}(t_\K)-u^\K_{i_\K}(t_\K))\geq lh_\K\partial_x\phi(t_\K,x_\K)+C_1(\varepsilon),$$
 where $C_1(\varepsilon)\rightarrow 0,$  as $\varepsilon \rightarrow 0$. Using similar arguments replacing  $l$ by $l+1$ in the above inequality and using  the Lipschitz estimates, we have 
$$\frac{1}{\K}( u^\K_{l+i_\K+1}(t_\K)-u^\K_{i_\K}(t_\K))\geq lh_\K\partial_x\phi(t_\K,x_\K)+C_2(\varepsilon),$$
where $C_2(\varepsilon)\rightarrow 0,$  as $\varepsilon \rightarrow 0.$
We set $C(\varepsilon)=\min(C_1(\varepsilon),C_2(\varepsilon)),$ and
 using \eqref{eq: interpolation inequality}-\eqref{eq:growth lipschitz}, we bound the remainder of the sum by zero to obtain 
\begin{align*}
    \partial_t \phi(t,x)&\geq R(x,\overline{I}(t))+\liminf_{\varepsilon\rightarrow0}\sum\limits_{\vert l \vert \leq \lfloor \frac{M}{h_\K}\rfloor}\left((1-\frac{x_\K}{\delta _\K}+i_\K)p((l+i_\K+1)\delta_\K)+ (\frac{x_\K}{\delta_\K}-i_\K)p((l+i_\K)\delta_\K)\right)
    \\
&~~~~~~~~~~~~~~~~~~~~~~~~~~~~~~~~~~~~~~~~~~~~~~~~~~~~~~~~~~~~~~~~~~~~~~~~~~~~~~~~~~~~~~h_\K G(lh_\K)e^{lh_\K \partial_x\phi(t_\K,x_\K)+C(\varepsilon)}.
\end{align*}
Moreover, since $p$ is Lipschitz continuous, we have  
$$ (1-\frac{x_\K}{\delta _\K}+i_\K)p((l+i_\K+1)\delta_\K)+(\frac{x_\K}{\delta_\K}-i_\K)p((l+i_\K)\delta_\K)\geq   p(x)-\Vert p\Vert_{Lip} (\vert l\vert +1)\delta_\K.$$
We deduce that 
\begin{align*}
    \partial_t \phi(t,x)\geq R(x,\overline{I}(t))+\liminf\limits_{\K\rightarrow 0}\sum\limits_{\vert l \vert \leq  \lfloor \frac{M}{h_\K}\rfloor}\left(p(x)-\Vert p\Vert_{Lip} (\vert l\vert +1)\delta_\K \right)h_\K G(lh_\K)e^{lh_\K\partial_x\phi(t_\K,x_\K)+C(\varepsilon)}.
\end{align*}
Let us now show that
\begin{equation}
\label{eq:nul part}
\lim\limits_{\K\rightarrow 0} \sum_{\vert l\vert \leq  \lfloor \frac{M}{h_\K}\rfloor}(\vert l\vert +1)\delta_\K h_\K G(lh_\K)e^{lh_\K\partial_x\phi(t_\K,x_\K)+C(\varepsilon)}=0,
\end{equation}
and that
\begin{equation}
\label{eq:hamiltonian part}
    \liminf\limits_{\K\rightarrow 0} \sum_{\vert l\vert \leq \lfloor \frac{M}{h_\K}\rfloor}h_\K G(lh_\K)e^{lh_\K\partial_x\phi(t_\K,x_\K)+C(\varepsilon)}\geq \int_{\vert y\vert \leq M}G(y)e^{\partial_x \phi(t,x) .y}dy.
    \end{equation}
We start by proving \eqref{eq:nul part}. Since $(t_\K,x_\K)$ converges, there exists $b>0$ such that $\vert\partial_x \phi(t_\K,x_\K) \vert\leq b.$
We have
\begin{align*}
   \sum_{ \vert l\vert \leq \lfloor \frac{M}{h_k}\rfloor } (\vert l\vert+1)\delta_\K h_\K G(lh_\K )e^{bh_\K\vert l\vert +C(\varepsilon)}
   &= \varepsilon \sum_{ \vert l\vert \leq \lfloor \frac{M}{l}\rfloor } (\vert lh_\K\vert+h_\K) h_\K G(lh_\K)e^{b\vert lh_\K\vert+C(\varepsilon)}.
\end{align*}
The Riemann sum $\sum_{ \vert l\vert \leq \lfloor \frac{M}{h_k}\rfloor } (\vert lh_\K\vert+h_\K) h_\K G(lh_\K)e^{b\vert lh_\K\vert+C(\varepsilon)}$ converges to $\int_{\vert y\vert \leq M}\vert y\vert G(y)e^{b\vert y \vert}dy.$ Thus $\varepsilon \sum_{ \vert l\vert \leq \lfloor \frac{M}{h_k}\rfloor } (\vert lh_\K\vert +h_\K) h_\K G(lh_\K)e^{b\vert lh_\K\vert+C(\varepsilon)}$ converges to zero. Hence \eqref{eq:nul part} is proved.\\
We now show  \eqref{eq:hamiltonian part}. Since $\phi$ is a $C^\infty$ function, for any  $\eta >0$ and sufficiently small $\K,$ we have  $$\vert \partial_x \phi(t_\K,x_\K)- \partial_x\phi(t,x)\vert \leq \eta. $$
 Therefore,
\begin{align*}
    \sum_{\vert l\vert \leq \lfloor\frac{M}{h_\K}\rfloor}h_\K G(lh_\K)e^{ lh_\K\partial_x\phi(t_\K,x_\K)+C(\varepsilon)}&\geq  \sum_{\vert l\vert \leq \lfloor\frac{M}{h_\K}\rfloor}h_\K G(lh_\K)e^{ lh_\K\partial_x \phi(t,x)-\eta M+C(\varepsilon)}.
\end{align*}
This is a Riemann sum which converges to $\int_{\vert y\vert \leq M}G(y)e^{\partial_x\phi(t,x)y-\eta M}dy.$
By dominated convergence, we let $\eta$ tend to zero and we deduce \eqref{eq:hamiltonian part}.
From \eqref{eq:nul part} and \eqref{eq:hamiltonian part}, and letting $M$ tend to infinity, we conclude that 
$$\partial_t\phi(t,x)\geq R(x,\overline{I}(t))+p(x)\int_{\mathbb{R}}G(y)e^{\partial_x\phi(t,x)y}dy.$$
We now prove that $1-\vert \partial_x \phi(t,x)\vert \geq 0.$ To do so, we require the following lemma.
\begin{lem}
    For any $(t,x,y)\in (0,+\infty)\times \mathbb{R}^2,$ we have 
    \begin{align}  
    \underline{u}(t,x+y)-\underline{u}(t,x)\leq \vert y \vert.
    \end{align}
    i.e $\underline{u}$ is $1$-Lipschitz continuous in space, uniformly in time.
\end{lem}
The proof of this lemma will be given after the end of the proof of Proposition 5.1.\\

Since $\underline{u}-\phi$ is  minimal at  $(t,x),$ by Lemma 5.2 we have for all $y\in \R$
$$\phi(t,x+y)-\phi(t,x)\leq \underline{u}(t,x+y)-\underline{u}(t,x)\leq \vert y \vert. $$
 We conclude that  $$\vert \partial_x \phi(t,x)\vert \leq 1.$$
 Therefore, we have proved that $\underline{u}$ is a viscosity supersolution of \eqref{eq:HJ} in $(0,+\infty)\times \mathbb{R}.$\\
\textbf{Step 2: Initial condition.} By Lemma 4.1 and the uniform convergence of $u^{\K,0}$ we deduce that $\underline{u}(0,x)=u^0(x)$ for all $x\in \mathbb{R}.$  This concludes the proof.
\end{proof}
\subsubsection*{Proof of Lemma 5.2.}  
By contradiction, we assume that there exist $(t,x,y_0)\in (0,+\infty)\times \mathbb{R}^2$ and $\alpha >0$ such that
\begin{equation*}   
\underline{u}(t,x+y_0)-\underline{u}(t,x)\geq \vert y_0\vert +\alpha.
\end{equation*}
Since $\underline{u}$ is lower semi-continuous, there exists $\gamma >0$ such that
$$ \underline{u}(t,x+y)\geq \underline{u}(t,x+y_0)-\alpha/4,~~~~~\forall y\in (y_0-\gamma, y_0+\gamma).$$
We deduce that
\begin{equation}
\label{eq:lipschitz contradiction}
    \underline{u}(t,x+y)-\underline{u}(t,x)\geq \vert y\vert +5/8\alpha,~\forall y\in \mathcal{A}_0:=(y_0-\min(\gamma,\alpha/8), y_0+\min(\gamma,\alpha/8)).
\end{equation}
Let $\phi \in C^1((0,+\infty)\times \mathbb{R}) $ be such that $(t,x)$ is a strict global minimum of $\underline{u}-\phi$ on $(0,+\infty)\times \R.$
Then, there exists a sequence $(t_\K,x_\K)$ such that $(t_\K,x_\K)$ is a global minimum of $\widetilde{u}^\K-\phi$ and
$(t_\K,x_\K)\rightarrow(t,x),$ and $\widetilde{u}^\K(t_\K,x_\K)\rightarrow \underline{u}(t,x),$ as $\K\rightarrow 0.$ Therefore, by definition of $\underline{u},$ we have for any $y\in \mathcal{A}_0$ and $\K$ sufficiently small
\begin{equation}
\label{eq:approximation of lipschitz contradiction}
\widetilde{u}^\K(t_\K,i_\K\delta_\K+y)\geq \underline{u}(t,x+y)-\alpha/4,~~~~\text{and}~~~-\widetilde{u}^\K(t_\K,x_\K)\geq -\underline{u}(t,x)-\alpha/4,
\end{equation}
where $i_\K=\lfloor \frac{x_\K}{\delta_\K}\rfloor.$ By \eqref{eq:lipschitz contradiction}-\eqref{eq:approximation of lipschitz contradiction}  and the Lipschitz estimates of $u^\K$, we obtain
\begin{equation}
\label{eq:control}
    \widetilde{u}^\K(t_\K,i_\K\delta_\K+y)-u^\K_{i_\K}(t_\K)\geq \vert y \vert +\alpha/8-C(x)\delta_\K,~\forall y\in \mathcal{A}_0,
\end{equation}
where $C(x)$ is the spatial Lipschitz constant of $u^\K$ in $[x-1,x+1].$ Using Assumption \ref{item:4}
we deduce that
\begin{align*}
    \sum_{l\in \mathbb{Z}}p((l+i_\K)\delta_\K)&h_\K G(lh_\K)e^{\frac{1}{\K}(u^\K_{l+i_\K}(t_\K)-u^\K_{i_\K}(t_\K))} \geq \underline{p} \sum_{\{l\in \mathbb{Z}/~l\delta_\K \in \mathcal{A}_0\}}h_\K G(lh_\K)e^{\vert l h_\K \vert +\alpha/(8\varepsilon)-C(x)h_\K}\\
    &\geq \underline{p} \sum_{\{l\in \mathbb{Z}/~l\delta_\K \in \mathcal{A}_0\}}h_\K f(l h_\K)e^{\alpha/(8\varepsilon)-C(x)h_\K}\\
    &\geq \underline{p}h_\K \min_{x\in \R}f(x)\big(\lfloor \frac{y_0+\min(\gamma,\alpha/8))}{\delta_\K} \rfloor-(\lfloor \frac{y_0-\min(\gamma,\alpha/8)}{\delta_\K} \rfloor+1)\big)e^{\alpha/(8\varepsilon)-C(x)h_\K}\\
    &\geq 2\underline{p} \min_{x\in \R}f(x) \min(\gamma,\alpha/8)\frac{1}{\K}e^{\alpha/(8\varepsilon)-C(x)h_\K}. 
\end{align*}
Proceeding similarly using \eqref{eq:control} and the Lipschitz estimates of $u^\K$, we obtain the same  bound  with $i_\K+1$ instead $i_\K$. Since $\widetilde{u}^\K-\phi$ has a minimum at $(t_\K,x_\K)$ with $t_\K>0,$ we deduce that
$$\partial_t \phi(t_\K,x_\K)\geq \partial_t \widetilde{u}^\K(t_\K,x_\K)\geq -\overline{R}+ 2\underline{p} \min_{x\in \R}f(x) \min(\gamma,\alpha/8)\frac{1}{\K}e^{\alpha/(8\varepsilon)-C(x)h_\K},$$ 
which converges to $+\infty$ when $\K$ goes to zero.  This is a contradiction with the fact that the sequence $(t_\K,x_\K)_\K$ converges to $(t,x)$.
\hfill $\blacksquare$
\subsection{Viscosity subsolution}
\begin{pro}
   Under Assumptions \ref{item:1}-\ref{item:7}, we have $\overline{u}=\underline{u},$ and that the function $\overline{u}$ is a viscosity subsolution for the equation \eqref{eq:HJ}.
\end{pro}
Since the Hamiltonian \eqref{eq:Hamiltonian} can take infinite values, we cannot prove directly that $\overline{u}$ is a viscosity subsolution. Note for example that if $\phi$ is a smooth function such that $\overline{u} - \phi$ attains a maximum at some point $(t,x)$ with $|\nabla \phi(t,x)| < 1$, then passing to the limit in the mutation term, as in the proof of Proposition 5.1, needs to control the remainder of the sum. This requires that $\|\nabla \phi(t,\cdot)\|_{L^\infty(\mathbb{R})} < 1$. 
We can assume this latter condition only if $\overline{u}$ is $1$-Lipschitz. Therefore, we first prove the equality $\underline{u} = \overline{u}$, and then use the fact that $\underline{u}$ is $1$-Lipschitz, which allows us to conclude that $\overline{u}$ is a subsolution.
\begin{pro}
We have for any $t\geq0,$  $\overline{u}(t,.)=\underline{u}(t,.).$ In addition $\overline{u}(0,.)=\underline{u}(0,.)=u^0(.).$
\end{pro}
This result provides the convergence of $\widetilde{u}^\K$  to a continuous function. To prove this result, we first regularize the supersolution and modify it to satisfy certain required properties. Then, we use it as a test function that we compare with  $\overline{u}.$
\begin{proof}
  For some technical reason  we replace $R$ by  $\widetilde{R}$ with $\widetilde{R}(.,.)=R(.,.)-(\overline{R}+\overline{p}).$ Then $\underline{u}-(\overline{R}+\overline{p})t$ is a supersolution of
  \begin{equation}
    \label{eq:HJ~}
    \begin{cases}
      \min( \partial_tu(t,x) -\widetilde{R}(x,I(t))-p(x)\int_{\mathbb{R}}G(y)e^{\nabla u(t,x).y}, 1-\vert \nabla u(t,x)\vert )=0,~~\forall (t,x)\in(0,+\infty)\times  \mathbb{R},\\
      u(0,.)=u^0(.).
      \end{cases}
    \end{equation}
  We start by modifying $\underline{u}-(\overline{R}+\overline{p})t$ at the initial time in the following way: for any $(t,x)\in [0,+\infty)\times \R,$ we define
\begin{equation}
\label{eq:transf at t0}
\underline{u}_\circ (t,x)=
    \begin{cases}
        \underline{u}(t,x)-(\overline{R}+\overline{p})t~~~\text{for}~t>0,\\
        \liminf\limits_{\underset{s>0}{s\rightarrow 0}  }\underline{u}(s,x)~~~\text{for}~t=0.
    \end{cases}
\end{equation}
Note that from Proposition 5.1 and the lower semi-continuity of $\underline{u},$ we obtain 
$$u^0(x)= \underline{u}(0,x)\leq \underline{u}_\circ (0,x).$$
Then, $\underline{u}_\circ$ is a viscosity supersolution of \eqref{eq:HJ~} also at $t=0.$\\ We will regularize $\underline{u}_{\circ}$ in several steps.\\
\textbf{Step 1: Lipschitz continuity in time:} We perform an inf-convolution to make it also Lipschitz continuous in time: for $\gamma>0,$ we define
 \begin{equation}
  \label{eq:inf-cov}
\underline{u}_{\circ.\gamma}(t,x)=\inf\limits_{s\in \R^+}\{  \underline{u}_\circ(s,x)+\frac{\vert t-s\vert^2 }{\gamma^2}\}.
\end{equation}
Note that from the lower bound in  \eqref{eq:semi lineair growth},  $\underline{u}_{\circ.\gamma}$ is well-defined. We show that this infimum is attained. For fixed $(t,x)\in (0,+\infty)\times \R,$ $\gamma>0,$ and $\eta>0,$ there exists $s_\eta\in \R^+$ such that $$\underline{u}_{\circ,\gamma}(t,x)\geq \underline{u}_{\circ}(s_\eta,x)+\frac{\vert t-s_\eta\vert^2}{\gamma^2}-\eta.$$
By \eqref{eq:semi lineair growth}, we obtain that
\begin{equation*} \frac{\vert t-s_\eta\vert^2}{\gamma^2} -\eta\leq \underline{u}_{\circ}(t,x)-\underline{u}_{\circ}(s_\eta,x)\leq C_1t-C_2s_\eta ,
\end{equation*}
  which implies that
\begin{equation*}
    (\frac{\vert t-s_\eta\vert }{\gamma}-\vert C_2\vert\gamma/2)^2-\eta\leq (C_1-C_2)t+(C_2\gamma)^2/4.
\end{equation*}
We deduce that $(s_\eta)_\eta$ is bounded uniformly in $\eta$ and $\gamma$. Then along subsequences of $\eta$, it converges to $s(t)\geq 0,$ and we have 
\begin{equation}
    \label{eq:minimum control}
    \vert t-s(t)\vert \leq C(t)\gamma,
\end{equation}
where $C(t)=\sqrt{(C_1-C_2)t+1}+1.$ Moreover, by the lower semi continuity of $\underline{u}_{\circ},$ we deduce 
\begin{equation*}
    \underline{u}_{\circ,\gamma}(t,x)\geq \underline{u}_{\circ}(s(t),x)+\frac{\vert t-s(t) \vert^2}{\gamma^2}.
\end{equation*}
Hence the infimum in \eqref{eq:inf-cov} is attained. Moreover, by \eqref{eq:minimum control} we deduce that $s(t)$ tends to $t$ when $\gamma$ tends to zero. Therefore, by the lower semi continuity of $\underline{u}_{\circ}$ and \eqref{eq:inf-cov}, we conclude that $\underline{u}_{\circ,\gamma}$   converges to $\underline{u}_\circ$ as $\gamma \rightarrow 0.$\\
We now prove that $\underline{u}_{\circ,\gamma}$ is a supersolution of a modified version of \eqref{eq:HJ~}. Let $\phi\in C^1((0,+\infty)\times \R)$ and let $(t_0,x_0)\in (0,+\infty)\times \R$ be a minimum of $\underline{u}_{\circ,\gamma}-\phi.$ Since the infimum in \eqref{eq:inf-cov} is attained, there exists $s(t_0)\geq 0,$ such that $(t_0,s(t_0),x_0)$ is a minimum of the following function 
$$(t,s,x)\mapsto \underline{u}_{\circ}(s,x)+\frac{\vert t-s\vert ^2}{\gamma^2}-\phi(t,x).$$
Since $\underline{u}_{\circ}$ is a supersolution of $\eqref{eq:HJ~}$ on $[0,+\infty)\times \R,$ we deduce that 
$$\frac{2(t_0-s(t_0))}{\gamma^2}-\widetilde{R}(x_0,\overline{I}(s(t_0)))-p(x_0)\int_{\R}G(y)e^{\nabla \phi(t_0,x_0).y}dy\geq 0.$$
Moreover, since $t_0$ is a minimum point of $t\mapsto \underline{u}_{\circ}(s(t_0),x_0)+\frac{\vert t-s(t_0)\vert ^2}{\gamma^2}-\phi(t,x_0),$ we deduce that $\partial_t\phi(t_0,x_0)=\frac{2(t_0-s(t_0))}{\gamma^2}.$ Then
\begin{equation}
\label{eq:inf-equ}
\partial_t\phi(t_0,x_0)-\widetilde{R}(x_0,\overline{I}(t_0))-p(x_0)\int_{\R}G(y)e^{\nabla \phi(t_0,x_0).y}dy\geq \widetilde{R}(x_0,\overline{I}(s(t_0)))-\widetilde{R}(x_0,\overline{I}(t_0)).
\end{equation} \\
\textbf{Step 2: Regularization to be in the domain of the Hamiltonian:}  Let us show that the $\mu$-Lipschitz continuous function $\mu \underline{u}_{\circ,\gamma} $ for $\mu\in(0,1),$ is a viscosity supersolution for a perturbed equation of \eqref{eq:inf-equ}. By \eqref{eq:inf-equ},  we deduce that the function  $\mu \underline{u}_{\circ,\gamma} $ satisfies in the viscosity sense 
 $$ \partial_t u(t,x)\geq \mu \widetilde{R}(x,\overline{I}(t))+\mu p(x)\int_{\R}G(y)e^{ \nabla u(t,x).y/\mu}dy+\mu\big(\widetilde{R}(x,\overline{I}(s(t)))-\widetilde{R}(x,\overline{I}(t))\big).$$
   Since, $e^{ \nabla u(t,x).y}\leq \mu e^{ \nabla u(t,x).y/\mu}+ 1-\mu,$ and $\widetilde{R}$ is non-positive, thus
\begin{equation}
\label{eq:mu-HJ}
    \begin{aligned}
     \partial_t u(t,x)\geq \widetilde{R}(x,\overline{I}(t))+p(x)\int_{\R}G(y)e^{\nabla u(t,x).y}dy-(1-\mu)\overline{p}+\mu\big(\widetilde{R}(x,\overline{I}(s(t)))-\widetilde{R}(x,\overline{I}(t))\big).
    \end{aligned}
    \end{equation}
\textbf{Step 3: Lower bound:} Since the  Hamiltonian is concave with respect to the gradient variable and $\mu \underline{u}_{\circ,\gamma}$ is a supersolution to \eqref{eq:mu-HJ}, and since any constant is a supersolution to \eqref{eq:mu-HJ}, we deduce that for any constant $B,$ the function $\underline{u}^B_{\circ,\gamma,\mu}=\max(-B,\mu\underline{u}_{\circ,\gamma})$ is a supersolution of \eqref{eq:mu-HJ} (see Appendix B for the proof).
Moreover, $\underline{u}^B_{\circ,\gamma,\mu}$ is Lipschitz-continuous in time and space, so it is differentiable almost everywhere,
and satisfies  \eqref{eq:mu-HJ} almost everywhere.\\
\textbf{Step 4: $C^1$ regularity:}
Let $\rho:\R^+\times \R\rightarrow \R^+$ be a smooth function with support in the unit ball such that $\int_{\R^+\times\R}\rho(t,x)dtdx=1.$ We define $\rho_\eta(.,.)=\frac{1}{\eta^2}\rho(\frac{.}{\eta},\frac{.}{\eta}),$ and $\widetilde{\underline{u}}^B_{\circ,\gamma,\mu,\eta}$, for any $(t,x)\in [0,+\infty)\times \R,$  by
$$\widetilde{\underline{u}}^B_{\circ,\gamma,\mu,\eta}(t,x)=\underline{u}^B_{\circ,\gamma,\mu}*\rho_\eta(t,x).$$
Since \eqref{eq:mu-HJ} holds almost everywhere, we have 
\begin{equation}
\label{eq:conv eq}
\begin{aligned}
  \partial_t \widetilde{\underline{u}}^B_{\circ,\gamma,\mu,\eta}(t,x)\geq \widetilde{R}(,I(,))*\rho_\eta(t,x)+\int_{\R^+\times \R} &\rho_\eta(t-s,x-z)p(z) \int_{\R}G(y)e^{\nabla \underline{u}^B_{\circ,\gamma,\mu}(s,z).y} dydsdz-(1-\mu)\overline{p}\\
  &+\mu\big(\widetilde{R}(.,\overline{I}(s(.)))-\widetilde{R}(.,\overline{I}(.))\big)*\rho_\eta(t,x).
  \end{aligned}
\end{equation}
Since $p$ is Lipschitz-continuous and $\rho_\eta$ has its support in the ball centered in $0$ with radius  $\eta,$ we have 
\begin{align*}
\int_{\R^+\times \R} \rho_\eta(t-s,x-z)p(z) \int_{\R}&G(y)e^{\nabla \underline{u}^B_{\circ,\gamma,\mu}(s,z).y} dydsdz\\
&\geq p(x)\int_{\R^+\times \R} \rho_\eta(t-s,x-z) \int_{\R}G(y)e^{\nabla \underline{u}^B_{\circ,\gamma,\mu}(s,z).y} dydsdz\\
  &-\Vert p \Vert_{Lip}\eta \int_{\R^+\times \R} \rho_\eta(t-s,x-z) \int_{\R}G(y)e^{ \nabla \underline{u}^B_{\circ,\gamma,\mu}(s,z).y} dydsdz. 
  \end{align*}
By Jensen inequality we have $$\int_{\R^+\times \R} \rho_\eta(t-s,x-z) \int_{\R}G(y)e^{ \nabla \underline{u}^B_{\circ,\gamma,\mu}(s,z).y} dydsdz\geq \int_{\R}G(y)e^{\nabla \widetilde{\underline{u}}^B_{\circ,\gamma,\mu,\varepsilon}(t,x).y} dy.$$
Moreover, we have 
$$\int_{\R^+\times \R} \rho_\eta(t-s,x-z) \int_{\R}G(y)e^{\nabla \underline{u}^B_{\circ,\gamma,\mu}(s,z).y} dydsdz\leq \int_{\R}G(y)e^{\mu \vert y\vert } dy=:C_\mu.$$
Then 
\begin{equation}
\label{eq:conv mutation term}
\int_{\R^+\times \R} \rho_\eta(t-s,x-z)p(z) \int_{\R}G(y)e^{\nabla \underline{u}^B_{\circ,\gamma,\mu}(s,z).y} dydsdz\geq p(x)\int_{\R}G(y)e^{\nabla \widetilde{\underline{u}}^B_{\circ,\gamma,\mu,\eta}(t,x).y} dy-C_\mu\Vert p\Vert_{Lip}\eta.
\end{equation}
Let us now estimate the term $(\widetilde{R}(.,\overline{I}(s(.)))-\widetilde{R}(.,\overline{I}(.)))*\rho_\eta(t,x).$
To this end, we define as in \cite{Sepideh} the set 
$$\mathcal{A}_\theta=\{\tau\in [t-\eta,t+\eta]/ ~\vert \overline{I}(s(\tau))-\overline{I}(\tau)\vert\leq \theta\},$$ 
where $\theta$ is a small positive parameter. We have 
\begin{align*}
    &\big(\widetilde{R}(.,\overline{I}(s(.)))-\widetilde{R}(.,\overline{I}(.))\big)*\rho_\eta(t,x)=\int_{\mathbb{R^+}\times \mathbb{R}}\rho_\eta(t-\tau,x-y)\big(\widetilde{R}(y,\overline{I}(s(\tau)))-\widetilde{R}(y,\overline{I}(\tau))\big)d\tau dy\\
&=\int_{\R}\int_{\mathcal{A}_\theta}\rho_\eta(t-\tau,x-y)\big(\widetilde{R}(y,\overline{I}(s(\tau)))-\widetilde{R}(y,\overline{I}(\tau))\big)d\tau dy \\
&~~~~~~~~~~~~~~~~~~~~~~~~~~~~~~~~~~~~~~~~~~~~~~~~~~~~+\int_{\R}\int_{\mathcal{A}^c_\theta}\rho_\eta(t-\tau,x-y)\big(\widetilde{R}(y,\overline{I}(s(\tau)))-\widetilde{R}(y,\overline{I}(\tau))\big)d\tau dy\\
&=F_1+F_2.
\end{align*}
We have 
$$\vert F_1\vert \leq A_1 \theta.$$
By the definition of $\rho_\eta$ we deduce that 
$$\vert F_2\vert \leq \frac{2\overline{R}}{\eta}\vert \mathcal{A}^c_\theta\vert.$$
 Moreover, $\overline{I}$ is non-decreasing and bounded, then we have at most $I_M/\theta$ jumps of amplitude $\theta.$ Therefore, there exist at most $I_M/\theta$ disjoint intervals  $[\tau_i,\tau_i+C\gamma]$ for some positive constant $C$ such that $I(\tau_i+C\gamma)-I(\tau_i)>\theta.$
 From \eqref{eq:minimum control}, we have $\vert s(\tau)-\tau\vert\leq C(\tau) \gamma\leq C(t+1)\gamma,$ which means that $s(\tau)\in [\tau-C(t+1)\gamma,\tau+C(t+1)\gamma],$
then  $$\mathcal{A}^c_\theta\subset \cup_{1\leq i\leq \frac{I_M}{\theta}}[\tau_i-C(t+1)\gamma,\tau_i+2C(t+1)\gamma].$$
Thus, 
$$\vert \mathcal{A}^c_\theta\vert  \leq \frac{3C(t+1)\gamma I_M}{\theta}.$$
We deduce that  
\begin{equation}
\label{eq:conv control}
    \big(\widetilde{R}(.,\overline{I}(s(.)))-\widetilde{R}(.,\overline{I}(.))\big)*\rho_\eta(t,x)\geq -A_1\theta -\frac{C'(t) \gamma}{\eta \theta},
    \end{equation}
where $C'(t)=6\overline{R}I_MC(t+1).$ 
Therefore from \eqref{eq:conv eq}, \eqref{eq:conv mutation term} and \eqref{eq:conv control} , we conclude that 
\begin{equation}
\begin{aligned}
    \partial_t \widetilde{\underline{u}}^B_{\circ,\gamma,\mu,\eta}(t,x)-\widetilde{R}(.,I(.))*\rho_\eta(t,x))-p(x)\int_{\R}G(y)e^{\nabla \widetilde{\underline{u}}^B_{\circ,\gamma,\mu,\eta}(t,x).y} dy\geq&- C_\mu\Vert p\Vert_{Lip}\eta -\overline{p} (1-\mu)\\
    &-A_1\theta -\frac{C'(t) \gamma}{\eta \theta}.
\end{aligned}
\end{equation}
Since $\vert \nabla \widetilde{\underline{u}}^B_{\circ,\gamma,\mu,\eta}\vert\leq \mu,$ then by the dominated convergence theorem and letting  $t$ go to zero, the above equation is satisfied up to $t=0.$\\
\textbf{Step 5: Strict supersolution:} Let $T>0.$ Let us define $\underline{u}^B_{\circ,\gamma,\mu,\eta}$ for a constant $K$ to be chosen later, for any $(t,x)\in[0,T]\times \R.$
$$\underline{u}^B_{\circ,\gamma,\mu,\eta}(t,x)=\widetilde{\underline{u}}^B_{\circ,\gamma,\mu,\eta}(t,x)+K(1-\mu)t.$$
We choose $K,\mu,\eta,\theta$ and $\gamma$ such that 
$$-C_\mu\Vert p\Vert_{Lip}\eta -A_1\theta -\frac{C'(T) \gamma}{\eta \theta}+(K-\overline{p} )(1-\mu)>0.$$ We obtain 
\begin{equation}
\label{eq:regularize supersolution}
    \partial_t \underline{u}^B_{\circ,\gamma,\mu,\eta}(t,x)> \widetilde{R}(.,I(.))*\rho_\eta(t,x))+p(x)\int_{\R}G(y)e^{\nabla \underline{u}^B_{\circ,\gamma,\mu,\eta}(t,x).y} dy~~\forall(t,x)\in [0,T]\times \R,
\end{equation}
and \begin{equation}
\label{eq:gradient bound}
    \vert \nabla \underline{u}^B_{\circ,\gamma,\mu,\eta}\vert\leq \mu.
\end{equation}

We now prove that $$\overline{u}\leq \underline{u}.$$
 Let $T>0.$ We want to prove that the supremum of $\overline{u}-(\overline{R}+\overline{p})t-\underline{u}^B_\circ$ on $[0,T]\times \R$ is non positive, where $\underline{u}^B_\circ=\max(-B,\underline{u}_\circ)$. By contradiction, we assume that \begin{equation}
 \label{eq:contradiction}
     \sup_{(t,x)\in [0,T]\times \R}\overline{u}(t,x)-(\overline{R}+\overline{p})t-\underline{u}^B_\circ(t,x)=a>0.
     \end{equation}
According to the upper bound in Lemma 4.1, and the fact that $\underline{u}^B_\circ\geq -B$, we deduce that the supremum is attained at some point $(t_0,x_0)\in [0,T]\times F$ with $F$ a compact set. 
 For $\gamma,\eta$ small enough, and $\mu$ close to $1,$ the function $\overline{u}-(\overline{R}+\overline{p})t-\underline{u}^B_{\circ,\gamma,\mu,\eta}$
has a positive maximum greater than $a/2$ at some point $(t_1,x_1)\in [0,T]\times F,$ we can take a larger compact set $F$ if necessary. To deal with the discontinuity in $\widetilde{R}$ we use the method of discontinuous viscosity solution  where the Hamiltonian is $L^1$ in time (see \cite{Ishi,PL}). In this respect, we define 
    $$g^\K_\eta(t)=\sup\limits_{x\in F} \Big(\widetilde{R}(x,I^\K(t))-\widetilde{R}(.,I^\K(.))*\rho_\eta(t,x)\Big),$$
$$g_\eta(t)=\sup\limits_{x\in F} \Big(\widetilde{R}(x,I(t))-\widetilde{R}(.,I(.))*\rho_\eta(t,x)\Big).$$
Since $\widetilde{R}$ is bounded and Lipschitz-continuous in $x$, and by the  almost everywhere convergence of $I^\K,$ and the uniform convergence of the convolution regularization, we deduce that   $$\int_{0}^{t}g^\K_\eta(s)ds \rightarrow  \int_{0}^{t}g_\eta(s)ds, ~ \text{as}~\K\rightarrow 0,$$
and  $$\int_{0}^{t}g_\eta(s)ds\rightarrow 0, ~~\text{as}~\eta\rightarrow 0.$$
Therefore for $\eta$ small enough $\overline{u}-(\overline{R}+\overline{p})t-\underline{u}^B_{\circ,\gamma,\mu,\eta}-\int_{0}^{t}g_\eta(s)ds$ attains a maximum at some point $(t,x)\in F.$
Then (see \cite{G}), there exist a subsequence of $\K$ and a sequence $(t_\K,x_\K)$ such that $\widetilde{u}^\K-(\overline{R}+\overline{p})t-\underline{u}^B_{\circ,\gamma,\mu,\eta}-\int_{0}^{t}g^\K_\eta(s)ds$ is maximal at $(t_\K,x_\K)$ and $\widetilde{u}^\K(t_\K,x_\K)$ converges to $\overline{u}(t,x)$ and $(t_\K,x_\K)$ converges to $(t,x),$ as $\K\rightarrow 0.$ \\
\textbf{Case 1:} $t>0.$ Then, for $\K$ sufficiently small, we have $t_\K>0,$ and we obtain that
\begin{equation}
\label{eq:inqp}
\partial_t\underline{u}^B_{\circ,\gamma,\mu,\eta}(t_\K,x_\K)+\overline{R}+\overline{p}+g^\K_\eta(t_\K)\leq \partial_t\widetilde{u}^\K(t_\K,x_\K)=(1-\frac{x_\K}{\delta _\K}+i_\K)\frac{d}{dt}u^\K_{i_\K}(t_\K)+(\frac{x_\K}{\delta_\K}-i_\K)\frac{d}{dt}u^\K_{i_\K+1}(t_\K),
\end{equation}
where $i_\K=\lfloor \frac{x_\K}{\delta_\K}\rfloor.$ Therefore, using the fact that $(t_\K,x_\K)$ is a maximum point, we have 
\begin{align*}
 \frac{d}{dt}u^\K_{i_\K}(t_\K)&=R(i_\K\delta_\K,I^\K(t_\K))+\sum_{l\in \Z}  p((l+i_\K)\delta_\K)h_\K G(lh_\K)e^{\frac{1}{\K}(u^\K_{l+i_\K}(t_\K)-u^\K_{i_\K}(t_\K))}\\
 &\leq R(i_\K\delta_\K,I^\K(t_\K))+\sum_{l\in \Z}  p((l+i_\K)\delta_\K)h_\K G(lh_\K)e^{\frac{1}{\K}(\underline{u}^B_{\circ,\gamma,\mu,\eta}(t_\K,(l+i_\K)\delta_\K)-\underline{u}^B_{\circ,\gamma,\mu,\eta}(t_\K,x_\K))+C(x)h_\varepsilon},
\end{align*}
where $C(x)$ is the Lipschitz estimates of $u^\varepsilon,$ in $[x-1,x+1].$
On the one hand, for $M>0$, we have 
\begin{align*}
    \sum_{\vert l h_\K\vert \leq M}  p((l+i_\K)\delta_\K)h_\K G(lh_\K)e^{\frac{1}{\K}(\underline{u}^B_{\circ,\gamma,\mu,\eta}(t_\K,(l+i_\K)\delta_\K)-\underline{u}^B_{\circ,\gamma,\mu,\eta}(t_\K,x_\K))+C(x)h_\varepsilon
    }\\
    \leq (p(x)+\varepsilon (M+1)\Vert p\Vert_{Lip})\sum_{\vert l h_\K\vert \leq M}h_\K G(lh_\K)e^{ lh_\K\nabla \underline{u}^B_{\circ,\gamma,\mu,\eta}(t_\K,x_\K)+C'(x)h_\varepsilon},
\end{align*}
where $C'(x)$ is a positive constant independent of $\K,$ and can depend on $\gamma,\mu$ and $\eta.$
Note that, for all $\beta>0,$  we have for $\K$ sufficiently small that $$\vert \nabla \underline{u}^B_{\circ,\gamma,\mu,\eta}(t_\K,x_\K) - \nabla \underline{u}^B_{\circ,\gamma,\mu,\eta}(t,x)\vert \leq \beta.$$
Moreover
$\sum_{\vert l h_\K\vert \leq M}h_\K G(lh_\K)e^{ lh_\K\nabla \underline{u}^B_{\circ,\gamma,\mu,\eta}(t,x)+\beta \vert lh_\K \vert +C'(x)h_\K}$ is a Riemann sum that converges to\\
$\int_{\vert y\vert \leq M}G(y)e^{\nabla \underline{u}^B_{\circ,\gamma,\mu,\eta}(t,x).y+\beta \vert y \vert}dy.$
Furthermore, by Assumption \ref{item:4}, we have 
\begin{align*}
  \sum_{\vert l h_\K\vert >M}p((l+i)\delta_\K)h_\K G(lh_\K)e^{\frac{1}{\K}(\underline{u}^B_{\circ,\gamma,\mu,\eta}(t_\K,(l+i_\K)\delta_\K)-\underline{u}^B_{\circ,\gamma,\mu,\eta}(t_\K,x_\K))+C(x)h_\varepsilon }&\leq  C \overline{p} \sum_{\vert l h_\K\vert >M} e^{-(1-\mu)\vert lh_\K\vert +C(x)h_\varepsilon}dy\\
  &\leq C \overline{p} \int_{\vert y\vert \geq M}e^{-(1-\mu)\vert y\vert +C(x)h_\varepsilon}dy.
\end{align*}
Since $\vert \nabla \underline{u}^B_{\circ,\gamma,\mu,\eta}(t,x)\vert \leq \mu,$ by the dominated convergence theorem, as $\beta $ tends to zero, we conclude that
\begin{align*}
    \limsup\limits_{\K\rightarrow 0}\sum_{ l \in \Z}  p((l+i_\K)\delta_\K)h_\K G(lh_\K)e^{\frac{1}{\K}(\underline{u}^B_{\circ,\gamma,\mu,\eta}(t_\K,(l+i_\K)\delta_\K)-\underline{u}^B_{\circ,\gamma,\mu,\eta}(t_\K,i_\K\delta_\K))+C(x)h_\K}\\
    \leq p(x)\int_{\vert y\vert \leq M}G(y)e^{\nabla \underline{u}^B_{\circ,\gamma,\mu,\eta}(t,x).y}dy+C\overline{p}\int_{\vert y\vert \geq M}e^{-(1-\mu)\vert y\vert}dy.
\end{align*}
We let $M$ go to infinity, we obtain 
\begin{equation}
\label{eq:inq1}
\begin{aligned}
\limsup\limits_{\K\rightarrow 0}\sum_{ l \in \Z}  p((l+i_\K)\delta_\K)h_\K G(lh_\K)e^{\frac{1}{\K}(\underline{u}^B_{\circ,\gamma,\mu,\eta}(t_\K,(l+i_\K)\delta_\K)-\underline{u}^B_{\circ,\gamma,\mu,\eta}(t_\K,i_\K\delta_\K))+C(x)h_\varepsilon}\\
    \leq p(x)\int_{\R}G(y)e^{\nabla \underline{u}^B_{\circ,\gamma,\mu,\eta}(t,x).y}dy.
\end{aligned}
\end{equation}
On the other hand, we have that
\begin{equation*}
\begin{aligned}
    R(i_\K\delta_\K,I^\K(t_\K))\leq \widetilde{R}(.,I^\K(.))*\rho_\eta(t_\K,i_\K \delta_\K)+\overline{R}+\overline{p}+g^\K_\eta(t_\K).  
    \end{aligned}
\end{equation*}
In addition $\widetilde{R}(.,I^\K(.))*\rho_\eta(,)$ is non-positive. Then, by Fatou's lemma and the pointwise convergence of $I^\K,$ we have that 
\begin{equation}
\label{eq:inq2}
    \limsup_{\K\rightarrow 0}\widetilde{R}(.,I^\K(.))*\rho_\eta(t_\K,i_\K \delta_\K)\leq \widetilde{R}(.,I(.))*\rho_\eta(t,x).
\end{equation}
In the same way, we show a similar bound if  $i_\K$ is replaced by $i_\K+1$.
We deduce from \eqref{eq:inqp},\eqref{eq:inq1} and \eqref{eq:inq2} that
$$\partial_t \underline{u}^B_{\circ,\gamma,\mu,\eta}(t,x)\leq \widetilde{R}(.,I(.))*\rho_\eta(t,x)+p(x)\int_{\R}G(y)e^{\nabla \underline{u}^B_{\circ,\gamma,\mu,\eta}(t,x) .y}dy.$$ 
 It is a contradiction with \eqref{eq:regularize supersolution}.\\
\textbf{Case 2} $t=0.$ 
If there exists a subsequence of $t_\K$ such that for $\K,$ small enough,   $t_\K=0,$ then $$\widetilde{u}^\K(t_\K,x_\K)=u^{\K,0}(i_\K\delta_\K)(1-\frac{x_\K}{\delta_\K}+i_\K)+u^{\K,0}((i_\K+1)\delta_\K)(\frac{x_\K}{\delta_\K}-i_\K).$$ Since $\overline{u}(0,x)=u^0(x)\leq \underline{u}^{B}_\circ(0,x),$ then for $\gamma$ and $\eta$  small enough, and $\mu$ close to $1,$ we have $$\overline{u}(0,x)-\underline{u}^B_{\circ,\gamma,\mu,\eta}(0,x)<a/2.$$ This is in contradiction with \eqref{eq:contradiction}. 
We can assume $t_\K>0,$ and treat this case in a similar way as in Case 1.\\
Therefore, we conclude that $\overline{u}(t,.)-(\overline{R}+\overline{p})t\leq \underline{u}_{\circ}^B(t,.)$ for all $t\geq 0.$
We let $B$ go to infinity to obtain $\overline{u}(t,.)-(\overline{R}+\overline{p})t\leq \underline{u}_{\circ}(t,.)$ for all $t\geq 0,$
where $\underline{u}_{\circ}(t,.)$ was defined in \eqref{eq:transf at t0}. We deduce that $\overline{u}(t,.)= \underline{u}(t,.)$ for all $t>0.$
All that remains is to verify the inequality at $t=0$. We use the second and third inequalities in \eqref{eq:semi lineair growth} of Lemma 4.1 and the uniform convergence of $u^{\K,0}$ to deduce that $$\overline{u}(0,x)=\underline{u}(0,x)=u^0(x).$$
\end{proof}

\subsubsection*{Proof of Proposition 5.3}
Let $\phi\in C^\infty((0,+\infty)\times \R)$ and let $(t_0,x_0)\in (0,+\infty)\times \R $ be a maximum point of $\overline{u}-\phi.$ If $1-\vert\nabla \phi(t_0,x_0)\vert\leq 0,$ we deduce that $\overline{u}$ is a viscosity subsolution of \eqref{eq:HJ} at $(t_0,x_0).$ Hence, we assume that $\vert\nabla \phi(t_0,x_0)\vert<1.$ Since $\overline{u}$ is 1-Lipschitz, without loss of generality we can assume that $\vert\nabla \phi(t_0,x)\vert<1,$ for all $x\in \R.$ In a similar way as in the proof of Proposition 5.4, we show that
$$\partial_t\phi(t_0,x_0)\leq R(x_0,\underline{
I}(t_0))+p(x_0)\int_{\R}G(y)e^{\nabla \phi(t_0,x_0).y}dy,$$
where $\underline{I}(t)=\liminf\limits_{s\rightarrow t}I(s).$ We conclude that $\overline{u}$ is a viscosity subsolution of \eqref{eq:HJ} in $(0,+\infty)\times \R.$ Moreover, by Proposition 5.4 we deduce that  $\overline{u}$ is a viscosity subsolution of \eqref{eq:HJ} in $[0,+\infty)\times \R.$
\hfill $\blacksquare$
\section{Proof of Theorem 2.1} From Propositions 5.1, 5.3, and 5.4, we conclude that $\widetilde{u}^\K$ converges to a continuous function $u=\overline{u}=\underline{u},$ which is a viscosity solution of the equation \eqref{eq:HJ}. We now prove the constraint $\max\limits_{x\in \mathbb{R}} u(t,x)=0.$ Using similar arguments as in \cite{BMP}, we assume by contradiction that there exists $(t_0,x_0)$ such that $u(t_0,x_0)=a>0.$ By continuity of $u$ there exists $r>0 $ such that $u(t_0,x)\geq a/2$ for all $x\in B(x_0,r).$  Hence, by the uniform convergence of $\widetilde{u}^\K,$ we have for all $i\in \mathbb{Z}$ such that $i\delta_\K\in B(x,r),$ that $n^\K_i(t_0)=e^{\frac{u^\K_i(t_0)}{\varepsilon}}\geq e^{\frac{a}{\varepsilon}}\rightarrow +\infty ,$ as $\K\rightarrow 0.$ Then $I^\K(t_0)\rightarrow +\infty,$ which is in contradiction with \eqref{eq:bound of totale size}, hence $\max_{x\in\R}u(t,x)\leq 0$. We prove now that $\max\limits_{x\in \mathbb{R}}u(t,x)\geq 0,~\forall t>0.$ Let $t>0.$ By the upper bound in Lemma 4.1, we have $u^\K_i(t)\leq -A\vert i\delta_\K\vert +B_1+C_1t,$ so for $M$ sufficiently large, we have 
\begin{align*}
I_m&\leq  \limsup\limits_{\K\rightarrow 0}\sum_{\{i\in \mathbb{Z}/\vert i\delta_\K \vert \geq M\}} \delta_\K e^{\frac{1}{\K}(-A\vert i\delta_\K\vert +B_1+C_1t)}+\limsup\limits_{\K\rightarrow 0}\sum_{\{i\in \mathbb{Z}/\vert i\delta_\K \vert < M\}}\delta_\K n^\K_i(t)\\
&\leq 
\limsup\limits_{\K\rightarrow 0}\sum_{\{i\in \mathbb{Z}/\vert i\delta_\K \vert < M\}}\delta_\K n^\K_i(t).
\end{align*}
If for all $\vert x\vert <M,$  $u(t,x)<0,$ we have $n^\K_i(t)=e^{\frac{u^\K_i(t)}{\varepsilon}}\rightarrow 0,$ as $\K\rightarrow0$ uniformly for all $i\in \mathbb{Z}$ such that $\vert i\delta_\K\vert <M.$ This is in contradiction with the above inequality.\\
Let us now prove the last part of Theorem 2.1. From \eqref{eq:bound of totale size} we have  $\Vert \widetilde{n}^\K\Vert_{L^\infty((0,+\infty),L^1(\R))}\leq 2I_M.$ Then  along a subsequence of $\K$ the sequence $(\widetilde{n}^\K)_\K$ has a weak limit measure $n$ in $L^\infty(w*(0,+\infty),\mathcal{M}^1(\R)).$
We now prove \eqref{eq:concentration}. Let $(t_0,x_0)\in (0,+\infty)\times \mathbb{R}$ such that $u(t_0,x_0)=-a<0.$ By the uniform convergence of $\widetilde{u}^\K,$ and the continuity of $u,$ there exists $\eta>0$ such that for $\K$ small enough, we have $u^\K_i(t)\leq -a/2$ for all $(t,i\delta_\K)\in [t_0-\eta,t_0+\eta]\times B(x_0,\eta).$ Hence, $$\widetilde{n}^\K(t,x)\leq e^{-\frac{a}{2\varepsilon}},~~\forall (t,x)\in [t_0-\eta,t_0+\eta]\times B(x_0,\eta).$$ 
Therefore, by the dominated convergence theorem, we obtain

$$\int_{[t_0-\eta,t_0+\eta]\times B(x_0,\eta)}n(t,x)dtdx=\int_{[t_0-\eta,t_0+\eta]\times B(x_0,\eta)} \lim\limits_{\K\rightarrow 0}\widetilde{n}^\K(t,x)dtdx=0.$$
Thus $n(t_0,x_0)=0.$
\newpage
\section*{Appendix A}
\subsection*{Proof of Theorem 3.2.}
We define the infinite real matrix $A^\K=(a^\K_{i,j})_{i,j\in \mathbb{Z}}$ where 
$a^\K_{i,j}=p(j\delta_\K)h_\K G((j-i)h_\K),$ and  the  infinite diagonal matrix $D^\K(n(t))$ whose diagonal element $d^\K_i(n(t))$ is given by\\ $d^\K_i(n(t))=R(i\delta_\K,\delta_\K\Vert n(t)\Vert_{\ell^1\mathbb{(\Z)}})$ for all $i\in \mathbb{Z},$ and we define the infinite vector  $n^\K=(n^\K_i)_{i\in \mathbb{Z}}.$ We write Equation \eqref{eq:rescaling model} as follows:
\begin{equation}
    \K\frac{d}{dt}n^\K(t)=(D^\K(n^\K(t))+A^\K)n^\K(t).
\end{equation}
Let $T>0.$ We consider the following closed subset of $C( [0,T],\ell^1(\mathbb{Z})):$
\begin{equation}
\mathcal{A}:=\{n\in  C( [0,T],\ell^1(\mathbb{Z})),~ \Vert n(t)\Vert_{\ell^1(\mathbb{Z})}\leq C(t) \Vert n^{\K,0}\Vert_{\ell^1(\mathbb{Z})},~\forall t\in [0,T]\},
\end{equation}
where $C(t):=1+e^{(\overline{R}+\overline{p}\alpha(0))t/\K},$ with $\alpha(.)$  the function defined in Lemma 3.1.\\
The space $C( [0,T],\ell^1(\mathbb{Z}))$ is a Banach space for the norm $$\Vert n\Vert_{L^{\infty}([0,T],\ell^1(\mathbb{Z}))}:=\sup_{t\in[0,T]}\sum_{i\in \mathbb{Z}}\vert n_i(t)\vert.$$
 We define the following mapping:
 \begin{align*}
 \Phi:& \mathcal{A}\mapsto \mathcal{A} \\
 &n \mapsto \Phi(n),
 \end{align*}
where $\Phi(n)$ is the unique solution of \begin{align*}
    \begin{cases}
       \K \frac{d}{dt}\Phi(n)(t)=(\widetilde{D}^\K(n(t))+A^\K)n(t),\\
       (\Phi(n)(0))_i=n^{\K,0}(i\delta_K),~~\forall i\in \Z,
    \end{cases}
\end{align*} 
with $\widetilde{D}^\K$ the infinite diagonal matrix with diagonal elements 
\begin{equation*}
   \widetilde{R}(i\delta_\K,I)= \begin{cases}
   R(i\delta_\K,I_m/2),~~~I\leq I_m/2,\\
    R(i\delta_\K,I),~~~I\in[I_m/2,2I_M],\\
        R(i\delta_\K,2I_M),~~~I\geq 2I_M.
    \end{cases}
\end{equation*}
We prove that $\Phi$ has a unique fixed point in $\mathcal{A}.$
Let us show that $\Phi$ is a mapping from   $\mathcal{A}$ to $\mathcal{A}.$
Indeed, let $n\in \mathcal{A}.$ Then, we have for any $i\in \mathbb{Z},$ and $t\in[0,T]$
\begin{align*}
 \K(\Phi(n)(t))_i&= \K n^{\K,0}(i\delta_\K)+\int_{0}^{t}(\widetilde{D}^\K(n(s))+A^\K)n(s))_ids\\
&=\K n^{\K,0}(i\delta_\K)+\int_{0}^{t}\Big(\widetilde{R}(i\delta_\K,\delta_\K \Vert n(s)\Vert_{\ell^1(\mathbb{Z})}) n_i(s)+ \sum_{j\in \mathbb{Z}}p(j\delta_\K)h_\K G((j-i)h_\K)n_j(s)\Big)ds. 
\end{align*}
 Then
\begin{align*}
      \Vert \Phi(n)(t)\Vert_{\ell^1(\mathbb{Z})}&\leq \Vert n^{\K,0}\Vert_{\ell^1(\mathbb{Z})}+\frac{1}{\K}\int_{0}^{t}\Big(\overline{R} \Vert n(s)\Vert_{\ell^1(\mathbb{Z})}+\overline{p}\sum_{i\in \mathbb{Z}} \sum_{j\in \mathbb{Z}}h_\K G((j-i)h_\K)n_j(s)\Big)ds\\
     &\leq (1+(\overline{R}+\overline{p}\alpha(0))C(t)t\frac{1}{\K})\Vert n^{\K,0}\Vert_{\ell^1(\mathbb{Z})}.
    \end{align*}
Therefore, for $T$ small enough, we have for all $t\in[0,T],$ that $\big(1+(\overline{R}+\overline{p}\alpha(0))C(t)t\frac{1}{\K}\big)\leq C(t).$ This yields $$\Vert \Phi(n)(t)\Vert_{\ell^1(\mathbb{Z})}  \leq  C(t) \Vert n^{\K,0}\Vert_{\ell^1(\mathbb{Z})},~~~\forall t\in[0,T].$$ 
Let us now show that $\Phi$ is a  contraction.\\ Let $n,m\in \mathcal{A},$
we have for any $i\in\mathbb{Z},$ and $t\in [0,T]$
\begin{align*}
    \K \Vert \Phi (n)(t)-\Phi (m)(t)\Vert_{\ell^1(\mathbb{Z})}&\leq\int_{0}^{t} \Big(\overline{R}+A_1\Vert n(t)\Vert_{\ell^1(\mathbb{Z})} )\sum_{i\in \mathbb{Z}} \vert n_i(s)-m_i(s)\vert +\overline{p}\sum_{i\in \mathbb{Z}} \sum_{j\in \mathbb{Z}}h_\K G((j-i)h_\K)\\
    &~~~~~~~~~~~~~~~~~~~~~~~~~~~~~~~~~~~~~~~~~~~~~~~~~~~~~~~~~~~~~~~~~~~~~~~~ \vert n_j(s)-m_j(s)\vert\Big) ds\\
    &\leq  (\overline{R}+A_1\Vert n(t)\Vert_{\ell^1(\mathbb{Z})}+ \overline{p}\alpha(0))T \Vert n-m\Vert_{L^{\infty}([0,T],\ell^1(\mathbb{Z}))}\\
    &\leq \big(\overline{R}+A_1C(t)\Vert n^{\K,0}\Vert_{\ell^1(\mathbb{Z})}+ \overline{p}\alpha(0)\big)T \Vert n-m\Vert_{L^{\infty}([0,T],\ell^1(\mathbb{Z}))}.
\end{align*}
 Note that, from \eqref{eq:sublinear growth initial conditon}, we have $\Vert n^{\K,0}\Vert_{\ell^1(\mathbb{Z})}\leq C_\K,$ for some positive constant $C_\K.$ Therefore, $\big(\overline{R}+A_1C(T)C_\K+ \overline{p}\alpha(0)\big)T/\K$ tends to zero as $T$ tends to zero.
Then, for $T$ small enough, $\Phi$ is a contraction in $\mathcal{A}$. Consequently $\Phi$ has a unique fixed point which is the unique solution for Equation \eqref{eq:rescaling model} in $[0,T]\times \Z$ with $\widetilde{R}$ instead of $R$.
By iteration in time, we can deduce that the solution exists globally.\\
   Let us now prove \eqref{eq:bound of totale size}. This proof is an adaptation of the proof of Theorem 2.4 in \cite{BMP}. 
We define for all $t>0,$ $J^\K(t)=\frac{d I^\K(t)}{dt},$
then 
\begin{align*}
    J^\K(t)&=  \sum_{i\in \Z}\delta_\K \frac{d}{dt}n^\K_i(t)
    \\ 
    &=\frac{1}{\K} \Big[\sum_{i\in \Z}\delta_\K \widetilde{R}(i\delta_\K,I^\K(t))n^\K_i(t)+\sum_{i\in \Z}\sum_{j\in \Z}\delta_\K p(j\delta_\K)h_\K G((j-i)h_\K)n^\K_j(t)\Big] \\
     &=\frac{1}{\K} \Big[\sum_{i\in \Z}\delta_\K n^\K_i(t)\big(\widetilde{R}(i\delta_\K,I^\K(t))+p(i\delta_\K)\sum_{j\in \Z} h_\K G(jh_\K) \big)\Big].
\end{align*}
If $I^\K(t)\geq I_M+C C(\K) $ where $C(\K)=\vert \sum_{j\in \Z} h_\K G(jh_\K)-1\vert,$ we deduce from Assumptions \ref{item:1} and \ref{item:3}, that for all $i\in \mathbb{Z}$
\begin{equation*}
\widetilde{R}(i\delta_\K,I^\K(t))+p(i\delta_\K)\leq -A_3C C(\K).
\end{equation*}
Then
\begin{align*}
\widetilde{R}(i\delta_\K,I^\K(t))+p(i\delta_\K)\sum_{j\in \Z} h_\K G(jh_\K)\leq (\overline{p}-A_3C)C(\K).
\end{align*}
Therefore, for $C> \overline{p}/A_3,$ we have 
$$\frac{dI^\K(t)}{dt}<0.$$
We conclude the proof of the upper bound. By similar arguments, we obtain  the lower bound. Hence Theorem 3.2 has been proved. 
\subsection*{Proof of Proposition 3.3}
    Using the notation and the computations in the above proof, 
    we have 
\begin{equation}
\label{eq:J}
    J^\K(t)=\frac{1}{\K} \Big[\sum_{i\in \Z}\delta_\K n^\K_i(t)\big(R(i\delta_\K,I^\K(t))+ p(i\delta_\K)\sum_{J\in \mathbb{Z}}h_\K G(jh_\K)\big)\Big].
\end{equation}
We set $C(\K)=\sum_{j\in\mathbb{Z}}h_\K G(jh_\K).$ Then
\begin{align*}
    \frac{dJ^\K(t)}{dt}&=\frac{1}{\K} \Big[\sum_{i\in \Z}\delta_\K n^\K_i(t)\partial_IR(i\delta_\K,I^\K(t))J^\K(t)
    +\delta_\K \frac{dn^\K_i(t)}{dt}\big( R(i\delta_\K,I^\K(t))+C(\K)p(i\delta_\K) \big)\Big]\\
    &=\frac{1}{\K} \Big[\sum_{i\in \Z}\delta_\K n^\K_i(t)\partial_IR(i\delta_\K,I^\K(t))J^\K(t)\Big]\\
    & \hspace{4cm} +\frac{1}{\varepsilon^2}\Big[\sum_{i\in \Z}\delta_\K( R(i\delta_\K,I^\K(t))+C(\K)p(i\delta_\K))R(i\delta_\K,I^\K(t))n^\K_i(t)\\
    &\hspace{2cm} +\sum_{i\in \Z} \delta_\K (R(i\delta_\K,I^\K(t))+C(\K)p(i\delta_\K))
    \sum_{j\in \Z} p(j\delta_\K)h_\K G((j-i)\delta_\K)n^\K_j(t)\Big]\\
    &=\frac{1}{\K} \sum_{i\in \Z}\delta_\K n^\K_i(t)\partial_IR(i\delta_\K,I^\K(t))J^\K(t)+\frac{1}{\varepsilon^2}\sum_{i\in \Z}\delta_\K (R(i\delta_\K,I^\K(t))+C(\K)p(i\delta_\K))^2\\
&~~~~n^\K_i(t)+\frac{1}{\varepsilon^2}\sum_{i\in \mathbb{Z}}\sum_{j\in \mathbb{Z}} \delta_\K (R(i\delta_\K,I^\K(t))+C(\K)p(i\delta_\K)-(R(j\delta_\K,I^\K(t))+C(\K)p(j\delta_\K)))\\
    &\hspace{6cm} p(j\delta_\K)h_\K G((j-i)h_\K )n^\K_j(t).
\end{align*}
Moreover, we have for any smooth function $h$
$$h(i\delta_\K)=h(j\delta_\K)+D_xh(j\delta_\K)(j-i)\delta_\K+\frac{((j-i)\delta_\K)^2}{2}D_{xx}h(x_{i,j}),$$
where $x_{i,j}\in \R.$
Then, using the fact that $G$ is an even function, Assumption \ref{item:2} and \eqref{eq:bound of totale size}, we deduce that for $\K$ small enough
\begin{align*}
\frac{1}{\varepsilon^2}\sum_{i\in \mathbb{Z}}&\sum_{j\in \mathbb{Z}} \delta_\K (R(i\delta_\K,I^\K(t))+C(\K)p(i\delta_\K)-(R(j\delta_\K,I^\K(t))+C(\K)p(j\delta_\K)))\\
    &p(j\delta_\K)h_\K G((j-i)\delta_\K)n^\K_j(t)\geq - I_M\overline{p}(\overline{R}+2\overline{p})\int_{\R}x^2G(x)dx.
 \end{align*}
Therefore,
\begin{equation}
    \frac{dJ^\K(t)}{dt}\geq -\frac{C}{\K}J^\K(t)-C_1,
\end{equation}
where $C$ and $C_1$ are positive constants. We define the negative part of $J^\K(t)$ as $(J^\K(t))_{-}=\max(0,-J^\K(t)).$
Thus,
    \begin{equation}
        \frac{d(J^\K(t))_{-}}{dt}\leq -\frac{C}{\K}(J^\K(t))_{-}+C_1.
    \end{equation}
  Moreover, by \eqref{eq:J}  we have 
  $\vert J^\K(0)\vert\leq \frac{C_2}{\K},$ for some positive constant $C_2.$ We deduce that 
  \begin{equation}
  \label{eq:negative part}
      (J^\K(t))_{-}\leq \frac{C_1}{C} \K+(J^\K(0))_{-}e^{-\frac{Ct}{\varepsilon}}\\
      \leq \frac{C_1}{C}\K+\frac{C_2}{\K}e^{-\frac{Ct}{\K}}.
 \end{equation}

We know that $\vert J^\K(t)\vert=J^\K(t)+2(J^\K(t))_{-}.$ Then for all $T>0,$ we have 
\begin{equation*}
   \int_{0}^{T} \Big\vert \frac{d I^\K(t)}{dt}\Big\vert dt =  \int_{0}^{T}  \frac{d I^\K(t)}{dt}+2\int_{0}^{T} \big(\frac{d I^\K(t)}{dt}\big)_{-}dt\leq 2I_M+2CT.
\end{equation*}
We conclude that the sequence $(I^\K)_{\K}$ has bounded variation, and we deduce the convergence almost everywhere along a subsequence of $\K$ to a function of bounded variation $I,$ which by \eqref{eq:negative part} is non-decreasing on $(0,+\infty).$ 
\hfill $\blacksquare$
\section*{Appendix B} 
\textbf{Proof of the maximum property for supersolutions with concave Hamiltonians.}\\
Here we state and prove the result for  time-independent Hamiltonians. The result is similar for Hamiltonians with time dependence of the form $b(t)+H(x,q),$ where $b\in L^1_{Loc}(\R^+),$ and no continuity assumption is required for $b$.
\begin{lem}
 Let $\Omega$ be an open convex subset of $\R$, and let $H: \R \times \Omega \to \R$ be a continuous Hamiltonian that is uniformly Lipschitz-continuous in space and concave with respect to the gradient variable. Let $u_1$ and $u_2$ be two Lipschitz-continuous in time and space, viscosity supersolutions of the problem 
\begin{equation}
\label{eq:maximum of viscosity super solution}
\partial_t u+H(x,Du)=0,
\end{equation} 
 such that almost everywhere $Du_i\in \Omega,$ for $i=1,2.$ Then, the function $v=\max(u_1,u_2)$  is also a viscosity supersolution of \eqref{eq:maximum of viscosity super solution}.
\end{lem}
\begin{proof}
Let $u_1$ and $u_2$ be two viscosity supersolutions of \eqref{eq:maximum of viscosity super solution}. 
We prove that $v=\max(u_1,u_2)$  is also a viscosity supersolution of \eqref{eq:maximum of viscosity super solution}. 
Using similar arguments as in the proof of Proposition 5.4, we regularize these solutions using a convolution approximation. We denote by $u_1^\eta$ and $u_2^\eta$ the regularized functions. By the uniform Lipschitz continuity and concavity of the Hamiltonian, we obtain for $i=1,2$ and for all $(t,x)\in (0,+\infty)\times \R$
\begin{equation}
\label{eq: perturbed eq}
  \partial_t u_i^\eta+H(x,D  u_i^\eta)\geq \int_{\mathbb{R}_+\times \R}\rho_\eta(t-s,x-y)(H(x,Du_i(s,y))-H(y,Du_i(s,y)))dsdy\geq o(1),
\end{equation}
where $\rho_\eta$ is a regularized function defined in Step 4 in the proof of Proposition  5.4 and $o(1)$ is a negligible constant of $\eta$ and independent of $(t,x)$ and $i.$ 
Let us prove  that $v^\eta=\max(u^\eta_1,u^\eta_2)$  is a viscosity supersolution of \eqref{eq: perturbed eq}.\\
Let $(t_0,x_0)\in (0,+\infty)\times \mathbb{R}$  and $(p,q)\in D_{-} v^\eta(t_0,x_0),$ where $D_{-}v^\eta$ is the subdifferential of $v^\eta$ (see \cite{G}). Then, there exists $\alpha \in [0,1]$ such that $$(p,q)=(\alpha \partial_t u^\eta_1(t_0,x_0)+(1-\alpha) \partial_t u^\eta_2(t_0,x_0),\alpha D u^\eta_1(t_0,x_0)+(1-\alpha) Du^\eta_2(t_0,x_0)).$$
The proof of this property is given after this proof. Then, by the concavity of the Hamiltonian, we obtain  
\begin{equation*}
    p+H(x_0,q)\geq \alpha( \partial_t u^\eta_1(t_0,x_0)+H(x_0,D u^\eta
    _1(t_0,x_0)))+(1-\alpha) (\partial_t u^\eta_2(t_0,x_0) +H(x_0,D u^\eta
    _2(t_0,x_0)))\geq o(1).
\end{equation*}
Thus, $v^\eta$  is a viscosity supersolution of \eqref{eq: perturbed eq}. Therefore, by the stability of viscosity solutions, we deduce that $v$ is a viscosity supersolution of \eqref{eq:maximum of viscosity super solution}.\\

We now prove that for $f=\max(f_1,f_2),$ where $f_1$ and $f_2$ are differentiable functions on $\mathbb{R}^2$, the subdifferential of $f$ satisfies at any point $x\in \mathbb{R}^2$
\begin{equation}
    D_{-}f(x)=\mathrm{Co}(Df_1(x),Df_2(x)),
\end{equation}
where $\mathrm{Co}(Df_1(x),Df_2(x))$ is the convex hull of $Df_1(x)$ and $Df_2(x).$\\
Let $x\in \mathbb{R}^2.$ The case where the maximum is defined by $f_1$ only or $f_2$ only is obvious, because if for example  $f_1(x)>f_2(x),$ we have $D_{-}f(x)=\{Df_1(x)\}.$\\
So we assume that $f_1(x)=f_2(x).$ By contradiction, we assume that there exists  $q\in D_{-}f(x)$ such that $q\notin \mathrm{Co}(Df_1(x),Df_2(x)).$ The set $\mathrm{Co}(Df_1(x),Df_2(x))$ is  convex and compact, so by the Hahn-Banach theorem, there exists $\xi\in \mathbb{R}^2$  such that
\begin{equation}
\label{eq:stricly separation}
\langle q,\xi\rangle > \langle y,\xi\rangle,~~~\forall y\in \mathrm{Co}(Df_1(x),Df_2(x)). 
\end{equation}
Let $\lambda>0$, we consider $z=x+\lambda \xi.$ Since $q\in D_{-}f(x),$ we have  for $\lambda$ small enough
\begin{align*}
    f(z)-f(x)\geq \langle q,z-x\rangle+o(\Vert z-x\Vert).
\end{align*}
We use the Taylor expansion at $x+\lambda \xi$ of $f_1$ and $f_2$ and let $\lambda$ tend to zero and the fact that $f(x)=f_1(x)=f_2(x),$ to deduce
\begin{equation*}
     \max(\langle Df_1(x),\xi\rangle,\langle Df_2(x),\xi\rangle)\geq \langle q,\xi \rangle. 
\end{equation*}
This is in contradiction with \eqref{eq:stricly separation}.
\end{proof}

\textbf{Acknowledgements:} The author is very grateful to Sepideh Mirrahimi  for  valuable discussions and continuous support in this work, to Nicolas Champagnat  and Sylvie Méléard for their fruitful discussions and comments on drafting this paper, and to Charles Bertucci for  fruitful discussions on maximum principle. This work is funded by the European Union (ERC, SINGER, 101054787). Views and opinions expressed are however those of the author only and do not necessarily reflect those of the European
Union or the European Research Council. Neither the European Union nor the granting authority
can be held responsible for them. This work has also been supported by the Chair "Modélisation Mathématique et Biodiversité" of Veolia Environnement-École Polytechnique-Museum National
d’Histoire Naturelle-Fondation X.

\newpage
\nocite{*}
\bibliographystyle{plain}
\bibliography{Rapport.bib}

\end{document}